\documentclass[11pt]{amsart}

\usepackage{geometry}
\usepackage{amsmath,amssymb,amsthm}
\usepackage{graphicx}
\usepackage{pst-plot}
\usepackage{hyperref}
\usepackage[active]{srcltx}
\usepackage{enumerate}
\usepackage{stmaryrd}

\newtheorem{theo}{Theorem}
\newtheorem{lem}{Lemma}

\newtheorem{algo}{Algorithm}
\newcommand{\eps}{\varepsilon}

\newcommand{\R}{\mathbb{R}}
\newcommand{\T}{\mathbb{T}^1}
\renewcommand{\eps}{\varepsilon}

\renewcommand{\o}{\overline}

\renewcommand{\d}{\delta}
\newcommand{\rd}{\mathrm{d}}
\newcommand{\x}{\mathbf{x}}
\newcommand{\V}{\mathbf{V}}
\newcommand{\D}{\Delta}
\numberwithin{equation}{section}
\begin{document}
\title[]{Numerical investigation of the Free Boundary regularity for a degenerate advection-diffusion problem}
\author{
L. Monsaingeon 
}
\address{CAMGSD Instituto Superior T\'ecnico, Av. Rovisco Pais
1049-001 Lisboa}
\email{leonard.monsaingeon@tecnico.ulisboa.pt}
\begin{abstract}
We study the free boundary regularity of the traveling wave solutions to a degenerate advection-diffusion problem of Porous Medium type, whose existence was proved in \cite{MonsaingonNovikovRoquejoffre}.
We set up a finite difference scheme allowing to compute approximate solutions and capture the free boundaries, and we carry out a numerical investigation of their regularity.
Based on some nondegeneracy assumptions supported by solid numerical evidence, we prove the Lipschitz regularity of the free boundaries.
Our simulations indicate that this regularity is optimal, and the free boundaries seem to develop Lipschitz corners at least for some values of the nonlinear diffusion exponent.
We discuss analytically the existence of corners in the framework of viscosity solutions to certain periodic Hamilton-Jacobi equations, whose validity is again supported by numerical evidence.
\end{abstract}
% Please provide minimum  5 keywords.
\maketitle
%
%
%%%%%%%%%%%%%%%%%%%%%%%%%%%%%%%%%%%%%%%%%%%%%%%%%%%%%%%%%%%%%%%%%%%%%%%%%%%%%%%%%%%%%%%%%%%%%%%%%%%%%%%%%%%%%%
%
\section{Introduction}
Consider the standard advection-diffusion equation
$$
\partial_t T-\mathrm{div}(\lambda \nabla T)+\mathbf V\cdot\nabla T=0,
\qquad t\geq 0,\,\x\in\R^d,
$$
with unknown temperature $T(t,\x)\geq 0$, conductivity $\lambda\geq 0$, and prescribed flow $\V(\x)\in\R^d$.
In the context of high temperature hydrodynamics the conductivity cannot be assumed to be constant as for standard diffusion, but should rather be of the form
$$
\lambda=\lambda(T)=\lambda_0 T^{m},\qquad
m>0
$$
for some fixed conductivity exponent and constant $\lambda_0>0$ depending on the model \cite{ZeldRaizer-physics}.
For example in physics of plasmas, and particularly in the context of Inertial Confinement Fusion, the dominant mechanism of heat transfer is the so-called electronic Spitzer heat diffusion and corresponds to $m=5/2$ (see e.g. \cite{AlmarchaClavinDucheminSanz-ablativeRT,ClavinMasse-instability}).
Suitably rescaling time and space yields the nonlinear parabolic problem
\begin{equation}
\partial_t T -\Delta\left(T^{m+1}\right) + \V\cdot \nabla T=0.
\label{eq:PMED}
\end{equation}
When $\V\equiv 0$ this corresponds to the celebrated the Porous Medium Equation
\begin{equation}
 \partial_tT-\D\left(T^{m+1}\right)=0,
\tag{PME}
\label{eq:PME_temp_parabolic}
\end{equation}
which has been widely studied in the literature as the basic example of a degenerate diffusion equation supporting finite speed of propagation.
We refer the reader to the classical monograph \cite{Vazq-PME} and references therein for an exhaustive bibliography, and to \cite{AronsonBenilan-regularite,AronsonCaff-initialtrace,BenilanPierre-solutionsPMERn} for well-posedness of the Cauchy problem and regularity questions.
Writing $\D T^{m+1}=\operatorname{div}((m+1)T^m\nabla T)$, it is clear that the diffusion degenerates at the levelset $\{T=0\}$ whenever $m>0$. 
In this setting it is well known \cite{DK07,Vazq-PME} that free boundaries $\Gamma=\partial\{T>0\}$
separate the ``hot'' region $D_+=\{T>0\}$ from the ``cold'' one $\{T=0\}$, and propagate with finite speed.
In order to study the propagation it is more convenient to use the pressure variable
\begin{equation*}
p=\frac{m+1}{m}T^{m},
\label{eq:defpressure}
\end{equation*}
which is well defined for physical temperatures $T\geq 0$ and formally satisfies
\begin{equation}
 \partial_tp-mp\D p+\V\cdot \nabla p=|\nabla p|^2.
\label{eq:PME_pressure_parabolic}
\end{equation}
Since $m>0$ and $p\propto T^m$, the degeneracy corresponds now to a vanishing ``coefficient'' $p=0$ along $\Gamma$ in the diffusion term.
As for most free boundary scenarios we cannot expect classical $\mathcal C^2$ solutions to exist, since gradient discontinuities may and usually do occur across the free boundaries.
A main difficulty is therefore to develop a suitable notion of viscosity and/or weak solutions, see e.g. \cite{CaffVazq-viscPME,CrandallIshiiLions-userguide,Vazq-PME} and references therein and thereof.
The parametrization, time evolution and regularity of the free boundary for \eqref{eq:PME_temp_parabolic} have been studied in details in \cite{CaffFried-regularity,CaffarelliVazquezWolanski-lipschitzPME,CaffWol-C1alpha}, and turns out to be a difficult question.
The case of potential flows $\V=\nabla \Phi$ has been studied in \cite{Kim}, where the authors investigate the long-time asymptotics of the free boundary for compactly supported solutions with an external confining potential ($\Phi(x)\propto |x|^2$ at infinity).
\\

In this article we focus on incompressible shear flows in the two dimensional periodic setting $\x=(x,y)\in D=\R\times \T$, and we consider
$$
\V(x,y)=\left(\begin{array}{c}
		\alpha(y)\\
		0
             \end{array}
\right),
\qquad \alpha(y+1)=\alpha(y),
\qquad \int_0^1\alpha(y)\mathrm{d}y=0
$$
for a sufficiently smooth $\alpha$ (the last condition is just a normalization).
Expressed in terms of the pressure variable, \eqref{eq:PMED} reads now in the cylinder
\begin{equation}
  \partial_tp-mp\D p+\alpha(y)\partial_xp=|\nabla p|^2,\qquad (t,x,y)\in\R^+\times\R\times\T.
\label{eq:PME-S_pressure_parabolic}
\end{equation}
Looking for traveling wave solutions $p(t,x,y)=p(x+ct,y)$ with speed $c>0$ yields the stationary PDE for the wave profiles
\begin{equation}
  -mp\D p+(c+\alpha(y))\partial_x p=|\nabla p|^2,
  \qquad (x,y)\in \R\times\T,
\label{eq:PME-S_pressure_stationary}
\end{equation}
and the following existence result was proved in \cite{MonsaingonNovikovRoquejoffre}:
\begin{theo}
Let $c_*:=-\min \alpha>0$.
For any $c> c_*$ there exists a continuous (very weak and viscosity) solution $p(x,y)\geq 0$ of \eqref{eq:PME-S_pressure_stationary} in the infinite cylinder $D=\R\times \mathbb T^{d-1}$ satisfying
\begin{enumerate}[(i)]
\item $p\in\mathcal{C}^{\infty}(D_+)$ and $0<\partial_x p\leq c_1$ in $D_+:=\{p>0\}$
\item $p$ is globally Lipschitz in $D$
\item $p$ is planar and linear at infinity in the direction of propagation $p(x,y)\sim cx$, $\partial_x p(x,y)\sim c>0$, and $p_y(x,y)\rightarrow 0$ uniformly in $y$ when $x\rightarrow +\infty$
\item
$p(x,y)\equiv 0$ for $x$ sufficiently negative
\end{enumerate}
The free boundary $\Gamma=\partial\{p>0\}\neq\emptyset$ can be parametrized as follows: there exists a bounded, periodic upper semi-continuous interface function $\mathcal I(y)$ such that
$$
p(x,y)>0\Leftrightarrow x>\mathcal I(y).
$$
Furthermore,
\begin{enumerate}[(a)]
\item
$\mathcal I$ is characterized by
\begin{equation}
 \mathcal I(y)=\inf\left(x\in\R,\quad p(x,y)>0\right)
\label{eq:char_I(y)}
\end{equation}
 \item 
If $y_0$ is a continuity point of $I$ then $\Gamma\cap\{y=y_0\}=(\mathcal I(y_0),y_0)\times\{y_0\}$. 
\item
If $y_0$ is a discontinuity point then $\Gamma\cap\{y=y_0\}=[\underline{\mathcal I}(y_0),\mathcal I(y_0)]\times\{y_0\}$, where $\underline{I}(y_0):=\displaystyle{\liminf_{y\rightarrow y_0}}\, \mathcal I(y)$.
\end{enumerate}
\label{theo:main}
\end{theo}

Let us stress that Theorem \ref{theo:main} was proved in \cite{MonsaingonNovikovRoquejoffre} in arbitrary dimensions $(x,y)\in\R\times\mathbb{T}^{d-1}$ ($d\geq 2$).
These solutions are the exact equivalent of the planar wave solutions
$$
p_c(x,y)=c[x]^+
$$
for the pure PME \eqref{eq:PME_pressure_parabolic} with $\V\equiv 0$, written here in the steady wave frame $x+ct$ and defined up to shifts ($[.]^+=\max\{.,0\}$ stands for the positive part).
Note that $c>0$ means propagation in the $x<0$ direction, and that any speed $c>0$ is admissible for the PME while $c>c_*$ in Theorem~\ref{theo:main}.
For these PME planar waves the interfaces in the moving frame are stationary $\mathcal C^{\infty}$ flat hypersurfaces $\{x=0\}$, up to shifts in the $x$ direction.
For general solutions to the PME, the interfaces tend to become or remain $C^{1,\alpha}$ depending on the initial regularity, see \cite{CaffFried-regularity,CaffarelliVazquezWolanski-lipschitzPME,CaffWol-C1alpha,daskalopoulos2003free}.

Note however that, when $\alpha(y)\not\equiv 0$, the parametrization function $\mathcal I(y)$ in Theorem~\ref{theo:main} is only upper semi-continuous.
At this stage the free-boundaries may have cusps, and we cannot exclude discontinuities of $\mathcal I(.)$ corresponding to horizontal segments $[\underline{\mathcal I}(y_0),\mathcal I(y_0)]\times\{y_0\}\subset \Gamma$ as in Theorem~\ref{theo:main}(c).
Due to the presence of the advection term, the by-now classical monotonicity techniques from \cite{alt1984variational,caffarelli2005geometric} do not immediately apply, and the purpose of this article is to investigate numerically (in dimension $d=2$) the free boundary regularity missing so far in Theorem~\ref{theo:main}.
\\

Using a finite difference scheme, we construct approximate solutions in truncated cylinders with suitable boundary conditions and track the free boundaries.
The simulations indicate that the solutions tend to grow linearly across the interface in the direction of propagation $x$, which is a strong nondegenerate behaviour and will be crucial in the subsequent analysis.
\emph{Assuming} this nondegeneracy, and under an additional regularity hypothesis, we will prove rigorously that the interface parametrization $\mathcal I(y)$ is Lipschitz continuous, solves a certain periodic Hamilton-Jacobi equation of the form
\begin{equation}
\label{eq:HJ_abstract}
|\nabla _y \mathcal I|^2=g(y),\qquad y\in \mathbb T^{d-1},
\end{equation}
and the free boundary is therefore the Lipschitz graph $\Gamma=\{x=\mathcal I(y)\}$.
Both assumptions will be supported by strong numerical evidence.
From the theory of periodic Hamilton-Jacobi equations we expect this Lipschitz regularity to be optimal, and this will be confirmed by our numerical experiments: depending on the value of the diffusion exponent, the interfaces appear to be regular when $m>1$, but seem to systematically develop Lipschitz corners for $m\in (0,1)$.
The value $m=1$ appears to be a sharp transition, at least in dimension $d=2$, and the very existence of corners is somewhat surprising: since our traveling waves are entire solutions to the parabolic problem \eqref{eq:PMED} one could expect smoothing as $t\to\infty$ and at least to some extent, even though the diffusion is degenerate.
This is in sharp contrast with the pure PME: in the original stationary $x$ frame the planar wave solutions $p_c(t,x,y)=c[x+ct]^+$ have of course smooth free boundaries, and for compactly supported initial data it is known \cite{daskalopoulos2001all,koch1998non,Vazq-PME} that the free boundaries become (spherical) $\mathcal C^\infty$ hypersurfaces $|x|\approx Ct^\beta$ asymptotically as $t\to\infty$ for some scaling exponent $\beta\equiv\beta(m,d)>0$.
In order to base here the existence of corners on more solid grounds than the mere observation of corner-\emph{looking} points in 2D plots obtained from simulations, we shall use in this paper the framework of Hamilton-Jacobi equations.
Indeed the classical theory for viscosity solutions \cite{Barles,CrandallIshiiLions-userguide,JensenSouganidis-regularity,Roquejoffre-PropQualHJ} tells us that the regularity of $\mathcal I(y)$ in \eqref{eq:HJ_abstract} is essentially related to the number of zeros of the right-hand side $g(y)$ in the torus.
This more tractable condition can be checked numerically, and will confirm the existence of corners in a more analytical context.

The paper is organized as follows: in section \ref{section:numerical_scheme} we introduce the discrete scheme, prove its positivity (Lemma~\ref{lem:positivity}), and highlight the numerical convergence in slowly drifting frames.
Section~\ref{section:nondegeneracy} contains the investigation of the nondegeneracy, based on numerical evidence.
Our main regularity result (Theorem~\ref{theo:regularity}) is proved in section \ref{section:regularity_corners}, which also contains the numerical validations as well as the interpretation in the framework of Hamilton-Jacobi equations.
%
%%%%%%%%%%%%%%%%%%%%%%%%%%%%%%%%%%%%%%%%%%%%%%%%%%%%%%%%%%%%%%%%%%%%%%%%%%%%%%%%%%%%%%%%%%%%%%%%%%%%%%%%%%%%%%
%
\section{Numerical scheme}
\label{section:numerical_scheme}
For the sake of simplicity we shall only consider the following three flows
$$
\begin{array}{ccc}
\alpha_1(y)  :=  0.5\sin(2\pi y),%\hspace{.2cm}
&
\alpha_2(y) :=  10\left(y^2(1-y)^2-\frac{1}{30}\right),%\hspace{.2cm}
&
\alpha_3(y) :=  \frac{1}{4}\sum\limits_{k=1}^4 \sin(2k\pi y)
\end{array}
$$
in our computations.
These are all normalized to be mean zero as required above, and $\alpha_3(y)$ is just a truncation of the Fourier expansion of a triangular sawtooth.

In order to approximate the wave profiles from Theorem~\ref{theo:main} we use a classical idea: traveling waves are usually attractors for the long-time dynamics of the associated Cauchy problems.
Fixing an admissible propagation speed $c>c_*$ as in Theorem~\ref{theo:main} (namely such that $c+\alpha(y)\geq c_0>0$) we work in the corresponding left-moving frame $x+ct$, in which \eqref{eq:PME-S_pressure_parabolic} reads
\begin{equation}
\partial_t p-mp\D p+ [c+\alpha(y)]\partial_x p=|\nabla p|^2.
\label{eq:PME-S_pressure_parabolic_waveframe}
\end{equation}
Starting with some suitable initial datum to be precised below, we expect a long-time convergence $p(t,x,y)\to \o p(x,y)$ of the Cauchy solution to the stationary wave profile satisfying \eqref{eq:PME-S_pressure_stationary}.
Since there exists a whole continuum of admissible speeds $c\in(c_*,+\infty)$, the speed selection by the long-time asymptotics is quite delicate. According to Theorem~\ref{theo:main}(iii) we know that the stationary wave profile satisfies $\partial_x\o p\sim c>0$ when $x\rightarrow +\infty$, and roughly speaking the slope at infinity determines the propagation speed.
This will be taken into account by imposing the Neumann boundary conditions $\partial_x p(t,.)=c$ ``at infinity'', or, rather, on the right boundary $x=X_{\max}$ of a large but finite computational cylinder $(x,y)\in [0,X_{\max}]\times\T$.
This is consistent with the construction in \cite{MonsaingonNovikovRoquejoffre}, where the solutions of Theorem~\ref{theo:main} were precisely obtained by solving the problem in truncated cylinders with suitable boundary conditions and letting the length of the cylinders tend to infinity.
%
%%%%%%%%%%%%%%%%%%%%%%%%%%%%%%%%%%%%%%%%%%%%%%%%%%%%%%%%%%%%%%%%%%%%%%%%%%%%%%%%%%%%%%%%%%
%
\subsection{Time and space discretization}
Choosing some large $X_{\max}>0$ and integers $N_x,N_y\in \mathbb N$, we work in the finite domain
$$
(x,y)\in D=[0,X_{\max}]\times\T
$$
and build a logically rectangular mesh $(i,j)\in \llbracket 1,N_x\rrbracket\times\llbracket 1,N_y\rrbracket$ as in Figure~\ref{fig:mesh}. 
Each cell is of size
$$
\rd x=\frac{X_{\max}}{N_x-1},\qquad \rd y=\frac{1}{N_y-1},
$$
and we denote the nodes by
$$
x_i=(i-1)\rd x,\qquad y_j=(j-1)\rd y,\qquad  (i,j)\in \llbracket 1,N_x\rrbracket\times\llbracket 1,N_y\rrbracket.
$$
\begin{figure}[!ht]
\begin{center}
\def\svgwidth{\columnwidth}
    \resizebox{0.8\textwidth}{!}{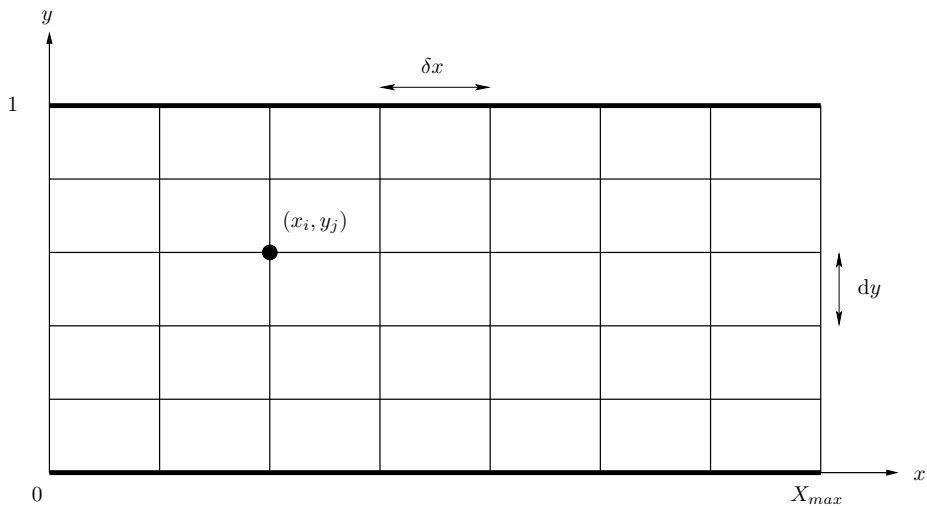}
\end{center}
\caption{Logically rectangular mesh. The top and bottom boundaries are identified through $y$-periodicity.}
\label{fig:mesh}
\end{figure}
\\
Since the Cauchy problem \ref{eq:PME-S_pressure_parabolic} is of course time-dependent, we also choose a large maximal time $T^{\max}$ and time steps
$$
0=t^0<...<t^n<t^{n+1}<...<T^N=T^{\max}, \qquad \rd t^n=t^{n+1}-t^n.
$$
For each iteration the adaptative time step will be chosen  in order to satisfy some Courant-Friedrichs-Lewy stability condition to be precised shortly.
We write as usual
$$
P^n_{i,j}\approx p(t^n,x_i,y_j),
$$
and the $y$-periodicity $j\equiv j\;\textrm{mod}\;(N_y-1)$ is used to compute $\partial_y$ derivatives on the top and bottom boundaries $j=1\equiv N_y$. 
In order to obtain an explicit scheme we approximate the time derivative by the forward difference
$$
\partial_t p(t^n,x_i,y_j)  \approx  \Delta^+_t P^n_{i,j} :=  \cfrac{P^{n+1}_{i,j}-P^n_{i,j}}{\rd t^n}.
$$
For the diffusion term we use the centered differences
\begin{eqnarray*}
\D p(t^n,x_i,y_k) & = & \left[\partial^2_{xx}p+\partial^2_{xx}p\right](t^n,x_i,y_j)\approx  \Delta^2_{xx} P^n_{i,j}+\Delta^2_{yy} P^n_{i,j},\\
\Delta^2_{xx} P^n_{i,j} & := & \cfrac{P^n_{i+1,j}+P^n_{i-1,j}-2P^n_{i,j}}{2\rd x^2},\\
\Delta^2_{yy} P^n_{i,j} & := & \cfrac{P^n_{i,j+1}+P^n_{i,j-1}-2P^n_{i,j}}{2\rd y^2}.
\end{eqnarray*}
Since we always assume $c>c_*=-\min\alpha\Leftrightarrow c+\alpha(y)\geq c_0>0$ in Theorem~\ref{theo:main}, we naturally use an upwind approximation for the advection term
\begin{equation}
(c+\alpha)\partial_xp(t^n,x_i,y_j)  \approx  [c+\alpha(y_j)]\D^-_x P^n_{i,j},
\qquad 
\D^-_x P^n_{i,j}  :=  \cfrac{P^n_{i,j}-P^n_{i-1,j}}{\rd x}.
\label{eq:upwinding_advection}
\end{equation}
Finally, we use centered differences for the right-hand side
\begin{eqnarray*}
|\nabla p|^2(t^n,x_i,y_j) & \approx & \left(\Delta_x P^n_{i,j}\right)^2+\left(\Delta_x P^n_{i,j}\right)^2,\\
\Delta_x P^n_{i,j} & := & \cfrac{P^{n}_{i+1,j}-P^{n}_{i-1,j}}{2\rd x},\\
\Delta_y P^n_{i,j} & := & \cfrac{P^{n}_{i,j+1}-P^{n}_{i,j-1}}{2\rd y}.
\end{eqnarray*}
Replacing each term in \eqref{eq:PME-S_pressure_parabolic_waveframe} by their discrete counterpart leads to the scheme
\begin{multline}
P^{n+1}_{i,j}=P^n_{i,j}+\rd t^n\Big{[}mP^n_{i,j}\left(\Delta^2_{xx} P^n_{i,j}+\Delta^2_{yy} P^n_{i,j}\right)\\
-\left(c+\alpha(y_j)\right)\Delta^-_x P^n_{i,j}+\left(\left(\Delta_x P^n_{i,j}\right)^2+\left(\Delta_y P^n_{i,j}\right)^2\right)\Big{]}.
\label{eq:def_scheme}
\tag{S}
\end{multline}
With this choice of discretization we expect first order accuracy in time and space, but we shall not investigate convergence orders in this work because no explicit solution is known so far.

From Theorem~\ref{theo:main}(i)-(iii) we know that the stationary solution should, at least qualitatively, resemble the classical PME planar wave $p_c(x,y)=c\left[x\right]^+$ up to translation in the $x$ direction.
We naturally use this profile as an initial condition
$$
p^0(x,y)=c\left[x-\tau\right]^+\quad \leadsto \quad P^0_{i,j}=c\left[x_i-\tau\right]^+,
$$
where the shift parameter $\tau\in(0,X_{\max})$ is chosen so that the initial free-boundary $\Gamma_0=\{x=\tau\}$ is well within the computational domain (say $\tau=X_{\max}/2$).

Since we are working in finite cylinders we also need to prescribe suitable boundary conditions on the sides $x\in \{0,X_{\max}\}$.
As stated in Theorem~\ref{theo:main}(iii) for the theoretical wave profile, the slope at infinity prescribes the propagation speed as $p(x,y)\sim cx$ and $\partial_x p\sim c$ when $x\to\infty$.
We consequently impose the Neumann condition on the right boundary $i=N_x$
\begin{equation}
\partial_xp(t,X_{\max},y)=c \qquad \leadsto \qquad  \frac{ P^n_{N_x,j}-P^n_{N_x-1,j}}{\rd x}=c, \quad j\in\llbracket1,N_y\rrbracket.
\label{eq:right_boundary_condition}
\end{equation}
As for the left boundary condition, let us recall that we chose an initial datum $p^0(x,y)=c[x-\tau]_+$ whose free boundary $\Gamma_0=\{x=\tau>0\}$ is away from $x=0$ at time $t=0$ (typically we use $\tau=X_{\max}/2$ with $X_{\max}$ large).
Because the free boundaries $\Gamma_t=\partial\{p(t,.)>0\}$ should propagate with finite speed \cite{Vazq-PME} and since the sought solution is stationary in the wave frame, we reasonably expect that $\Gamma_t$ should stay away from $x=0$ for later times $t>0$.
Thus the left boundary of the domain should therefore never ``see'' the solution $p(t,0,y)\equiv 0$, and we apply now the homogeneous Dirichlet condition
\begin{equation}
p(t,0,y)\equiv 0
\qquad\leadsto\qquad
P^n_{1,j}=0,\quad j\in\llbracket1,N_y\rrbracket.
\label{eq:left_boundary_condition}
\end{equation}
This leads to
\begin{algo}[Numerical solver for the Cauchy problem] Initialize $t^0=0$ and $P^0_{i,j}=c\left[x_i-\tau\right]^+$.
\label{algo:scheme}
\begin{enumerate}
\item
\label{item:update_tn}
For $n\geq 0$ choose
\begin{equation}
\rd t^n\leq \min\cfrac{1}{2(1/\rd x^2+1/\rd y^2)m\max\limits_{i,j}P^n_{i,j}+(c+\|\alpha\|_\infty)/\rd x}
\label{eq:def_dtn}
\end{equation}
and update $t^{n+1}:=t^n+\rd t^n$.
\item
Update $P^n\rightarrow P^{n+1}$ applying first \eqref{eq:def_scheme} in the interior $i\in\llbracket 2,N_x-1\rrbracket$, and then \eqref{eq:left_boundary_condition}\eqref{eq:right_boundary_condition} at the boundaries $i=1,N_x$.
\item
\label{item:algo_positivity}
If $P^{n+1}_{i,j}<0$, replace by $P^{n+1}_{i,j}:=0$.
\item
While $t^{n+1}<T^{\max}$, repeat from \ref{item:update_tn}.
\end{enumerate}
\end{algo}
\noindent
A typical computation is shown in figure \ref{fig:example_cauchy}: the pressure $p(t,x,y)$ evolves according to \eqref{eq:PME-S_pressure_parabolic} and the initially flat free boundary adjusts in time.

The first expression in the denominator of \eqref{eq:def_dtn} corresponds to the exact optimal CFL condition $\rd t\leq \rd t_{diff}=\cfrac{\rd x^2 \rd y^2}{2\lambda(\rd x^2 + \rd y^2)}=\mathcal O(\rd x^2+\rd y^2)$ for explicit schemes that one would get when considering the diffusion part $\partial_t p-mp\D p=0$ in \eqref{eq:PME-S_pressure_parabolic} as a linear diffusion equation $\partial_t p-D\Delta p=0$, with a diffusion coefficient $D=m p(t,x,y)$ and $\lambda = \|D\|_\infty$.
The second term in the denominator corresponds to the usual stability condition $\rd t\leq \rd t_{adv}=\rd x/\|\mathbf V\|_\infty$ for the advection part $\partial_t p =-(c+\alpha(y))\partial_x p$ in\eqref{eq:PME-S_pressure_parabolic_waveframe} with the upwinding \eqref{eq:upwinding_advection}.
In our simulations $P^{n}_{i,j}$ remains of order one at least on the right boundary, thus the diffusive (quadratic) CFL condition takes over the hyperbolic (linear) condition and in practice \eqref{eq:def_dtn} always selects $\rd t^n\approx \rd t_{CFL}^n=\mathcal O(\rd x^2+\rd y^2)$.
Moreover from our numerical experiments the quasilinear diffusion $-mp\D p$ seems to provide sufficient stability and \eqref{eq:def_dtn} seems optimal, in the sense that the scheme appears to be stable when this CFL condition is enforced, whereas instabilities started building up when $\rd t^n>\rd t_{CFL}$.
We do not pretend here to prove any rigorous stability/convergence results, but rather observe numerical stability from our computations.

It is worth stressing that \eqref{eq:PME-S_pressure_parabolic_waveframe} satisfies a comparison principle at the continuous level: nonnegative initial data should therefore produce nonnegative solutions, and step \ref{item:algo_positivity} accordingly prevents numerical errors from producing undesired negative values.
This truncation is actually never performed during the computations, and in fact the scheme is positive:
\begin{lem}
\label{lem:positivity}
Assume that $P_{i,j}^n\geq 0$ and that the time step satisfies \eqref{eq:def_dtn}.
Then $P_{i,j}^{n+1}\geq 0$ as well.
\end{lem}
\begin{proof}
Using $|\Delta_x P^n_{i,j}|^2+|\Delta_y P^n_{i,j}|^2\geq 0$, the interior scheme \eqref{eq:def_scheme} gives
\begin{multline*}
 P^{n+1}_{i,j}\geq P_{i,j}^n+\rd t\left[mP^n_{i,j}\left(\Delta^2_{xx} P^n_{i,j}+\Delta^2_{yy} P^n_{i,j}\right)
-\left(c+\alpha(y_j)\right)\Delta^-_x P^n_{i,j}\right]\\
= \beta^n_{i,j}P^n_{i,j}+ \sum\limits_{|k-i|+|l-j|=1}\beta^n_{k,l}P^n_{k,l}
\end{multline*}
for $i\in \llbracket 2,N_x-1\rrbracket$ and coefficients $\beta^n_{i,j}$.
For the off-diagonal indexes in the last sum, the consistent discretization of the diffusion terms and the upwinding \eqref{eq:upwinding_advection} automatically guarantee that $\beta_{k,l}\geq 0$.
Moreover with our CFL condition \eqref{eq:def_dtn} the diagonal coefficient reads
$$
\beta^n_{i,j}=1-\rd t^n\left[\left(\frac{2}{\rd x^2}+\frac{2}{\rd y^2}\right)m P^n_{i,j}+\frac{c+\alpha(y_j)}{\rd x}\right]\geq 0.
$$
Thus $P^{n+1}_{i,j}\geq 0$ as a positive linear combination of the nonnegative $P^n_{k,l}$'s, at least in the interior $(i,j)\in \llbracket 2,N_x-1\rrbracket\times \llbracket 1,N_y\rrbracket$.
On the left boundary $i=1$ we set the Dirichlet condition $P^{n+1}_{1,j}=0$ so our statement trivially holds.
Finally on the right boundary the Neumann condition immediately gives $P^{n+1}_{N_x,j}=P^{n+1}_{N_x-1,j}+c\,\rd x\geq P^{n+1}_{N_x-1,j}\geq 0$ and the proof is achieved.
\end{proof}
Typical sufficient conditions for the convergence of such parabolic schemes are stability, consistency, and monotonicity \cite{BS91,CL96}.
The positivity Lemma~\ref{lem:positivity} goes of course in the right direction, but for the sake of simplicity we shall not look any further into the discrete properties of the scheme.
\par
We implemented Algorithm~\ref{algo:scheme} in \texttt{Fortran90}, and all the simulations presented here were carried out on dedicated servers at the Institut de Math\'ematiques de Toulouse, France.
\begin{figure}[h!]
\begin{center}
\begin{tabular}{cc}
\includegraphics[width=7cm]{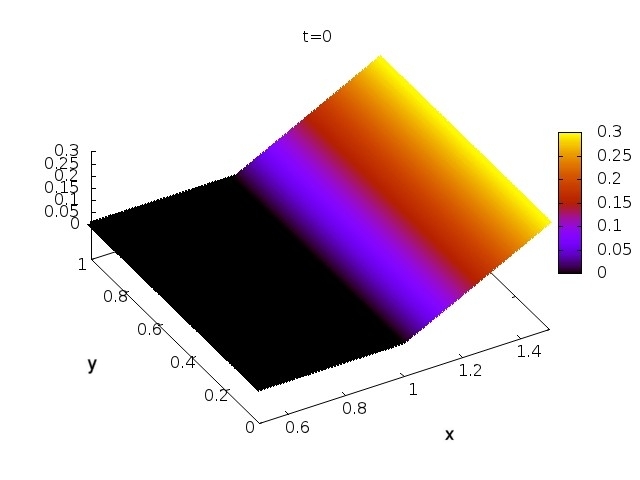} &
\includegraphics[width=7cm]{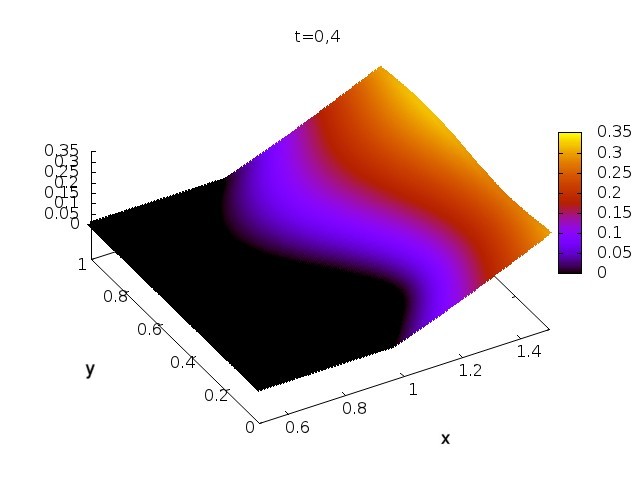}\\
\includegraphics[width=7cm]{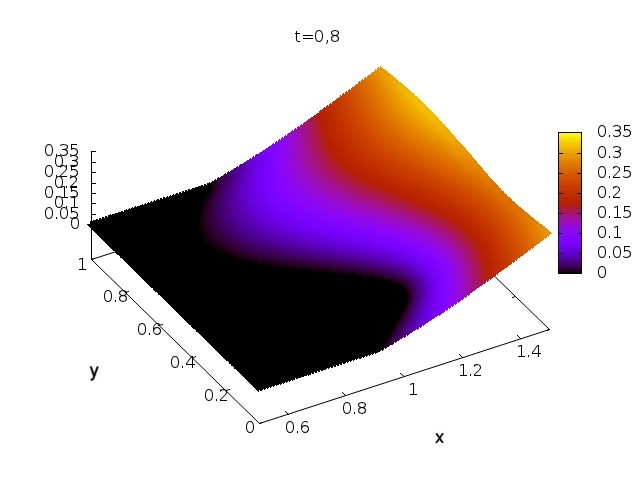} &
\includegraphics[width=7cm]{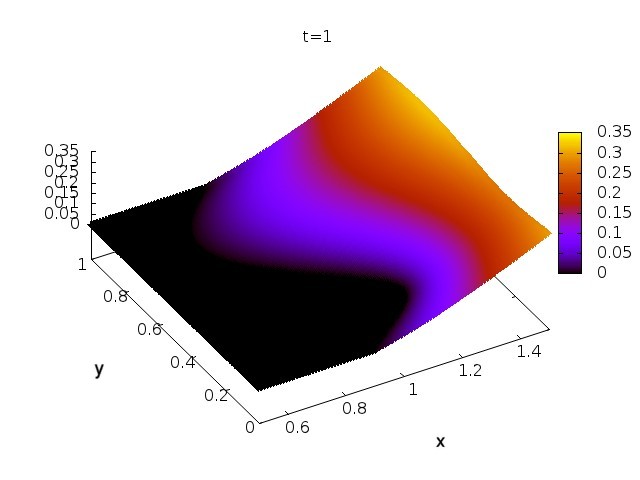}%\\
\end{tabular}
\end{center}
\caption{Snapshots of $p(t,.)$ plotted for $(x,y)\in[0.5,1.5]\times\T$ with parameters $m=1.1$, $\alpha(y)=\alpha_1(y)$, $c=0.6$, $X_{\max}=10$, $\rd x=\rd y =5.e^{-3}$.}
\label{fig:example_cauchy}
\end{figure}
%
%%%%%%%%%%%%%%%%%%%%%%%%%%%%%%%%%%%%%%%%%%%%%%%%%%%%%%%%%%%%%%%%%%%%%%%%%%%%%%%%%%%%%%%%%%%%%%%%%%%%%%%%%%%%%%
%
\subsection{Long-time convergence and slow drift}
\label{subsection:long-time_convergence}
As just discussed we aim at computing the stationary wave profile as the long-time asymptotic of the Cauchy solutions in the wave frame $x+ct$, and we applied Neumann condition $\partial_xp(t,X_{\max},y)=c$ on the right boundary in order to mimic the ``slope=speed'' behaviour in Theorem~\ref{theo:main}(ii).
However, since we necessarily compute on finite domains $x\in [0,X_{\max}]$, there is an inevitable discrepancy between the numerical paradigm on finite domains $\partial_x p(t,X_{\max},y)=c$ and the theoretical model $\partial_x p(+\infty,y)=c$.
As a result we cannot really expect any long-time convergence, even at the continuous level, and the finiteness of the domain will always lead to some residual error.
This is illustrated in figure \ref{fig:erreur_L2}: we see that
$\|\partial_tp(t,.)\|_{L^2}\rightarrow C_2>0$ and $\|\partial_tp(t,.)\|_{L^{\infty}}\rightarrow C_{\infty}>0$ as $t\rightarrow +\infty$, compared to the expected $\partial_tp\rightarrow 0$ for any long-time convergence.
\begin{figure}[!ht]
\begin{center}
\begin{tabular}{cc}
\includegraphics[width=7cm]{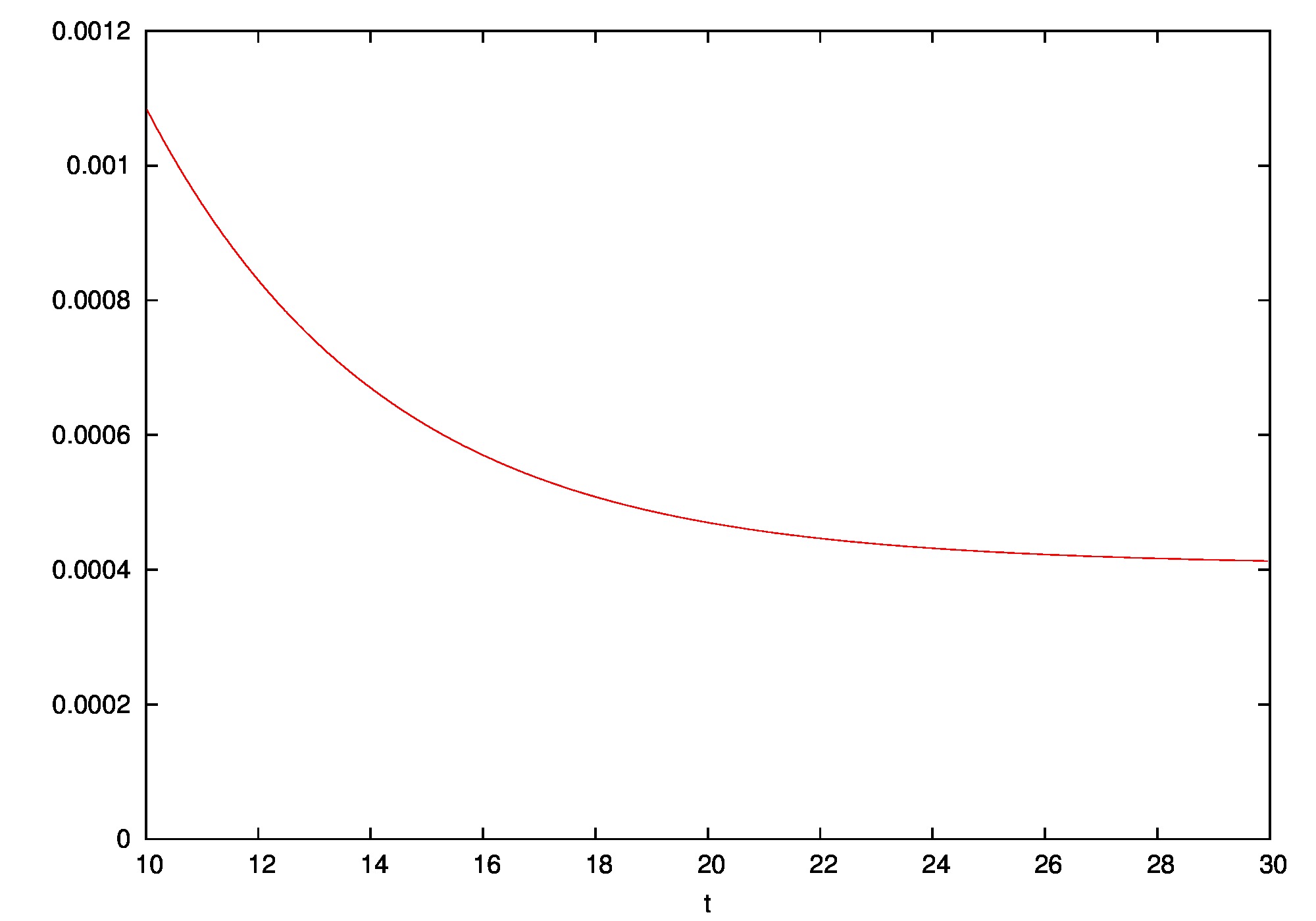}& \includegraphics[width=7cm]{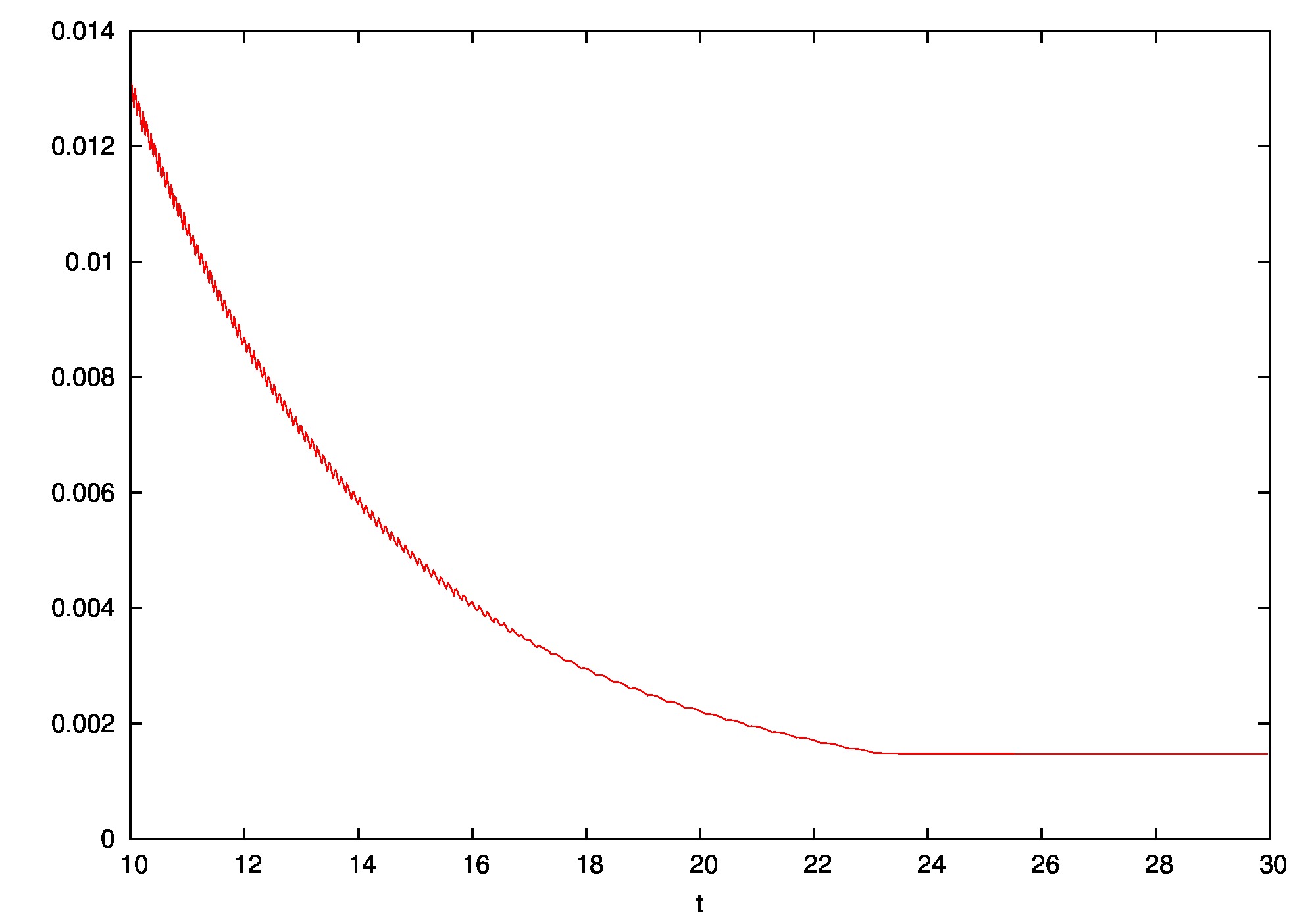}
\end{tabular}
\end{center}
\caption{Long-time asymptotics and residual error. $\|\partial_tp(t,.)\|$ plotted versus time in the $L^2$ (left) and $L^{\infty}$ (right) norms, with parameters $m=0.1$, $\alpha(y)=\alpha_2(y)$, $c=0.4$, $\rd x=\rd y=5.e^{-3}$}
\label{fig:erreur_L2}
\end{figure}

A natural but heuristic explanation, which we observed from our simulations, is the following: since the difference between the numerical and theoretical models comes from $X_{\max}<+\infty$, the numerical solution tends to globally shift in the $x<0$ direction in order to drive away from the right boundary and accommodate for the discrepancy. 
If $\overline{p}(x,y)$ denotes the stationary wave profile, this means that we expect the ansatz
\begin{equation}
p(t,x,y)\approx \overline{p}\left(x+X^*(t),y\right).
\label{eq:ansatz_p_shift}
\end{equation}
In order to determine and measure numerically the shift $X^*(t)$ we choose an arbitrary $y_0\in \T$ and monitor the quantity
\begin{equation}
\tilde{p}(t):=p(t,X_{\max},y_0),
\label{eq:def_p_tilda(t)}
\end{equation}
which can be numerically evaluated.
Indeed, this marker should evolve as
$$
\frac{d \tilde{p}}{d t}(t)=\partial_tp(t,X_{\max}, y_0)\approx \partial_x \overline{p}\left(X_{\max}+X^*(t),y_0\right) \frac {dX^*}{dt}(t).
$$
Since $X_{\max}$ is chosen large and $X^*(t)$ increases we expect the aforementioned ``slope=speed'' behaviour $\partial_x\overline{p}\left(X_{\max}+X^*(t),y_0\right)\approx \partial_x\overline p(\infty,y_0)=c>0$, and the shift $X^*(t)$ can thus be computed approximately as
\begin{equation}
\frac {dX^*}{dt}(t)\approx \frac{1}{c}\frac {d\tilde p}{dt}(t).
\label{eq:Xdot_heuristc_copmutation}
\end{equation}
In our computations $X^*(t)$ grows almost linearly in time but very slowly (typical values were $dX^*/dt\approx\mathcal O(10^{-3})$ over simulated $30-100$s times), and the solution accordingly drifts to the left.

The ansatz \eqref{eq:ansatz_p_shift} also suggests that one should in fact look for long-time convergence in the slowly moving frame $x+X^*(t)$, rather than in the fixed computational frame of reference.
In order to do so we observe from \eqref{eq:ansatz_p_shift} that
$$
\partial_t p(t,x,y)\approx \partial_x \overline{p}(x+X^*(t),y)\frac {dX^*}{dt}(t)\approx \partial_x p(t,x,y)\frac{d X^*}{d t}(t).
$$
\begin{figure}[h!]
\begin{center}
\begin{tabular}{cc}
\includegraphics[width=7cm]{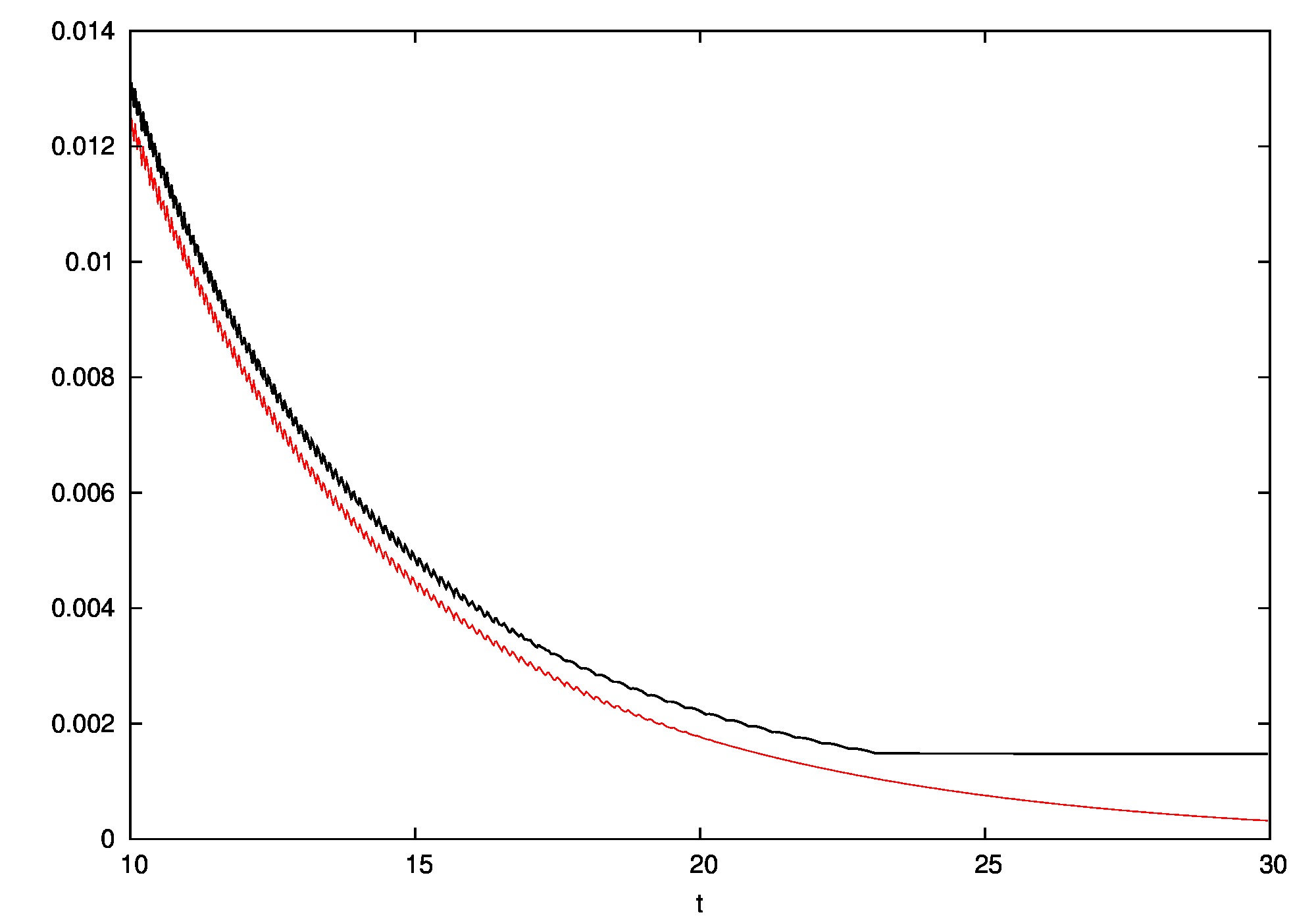}&
\includegraphics[width=7cm]{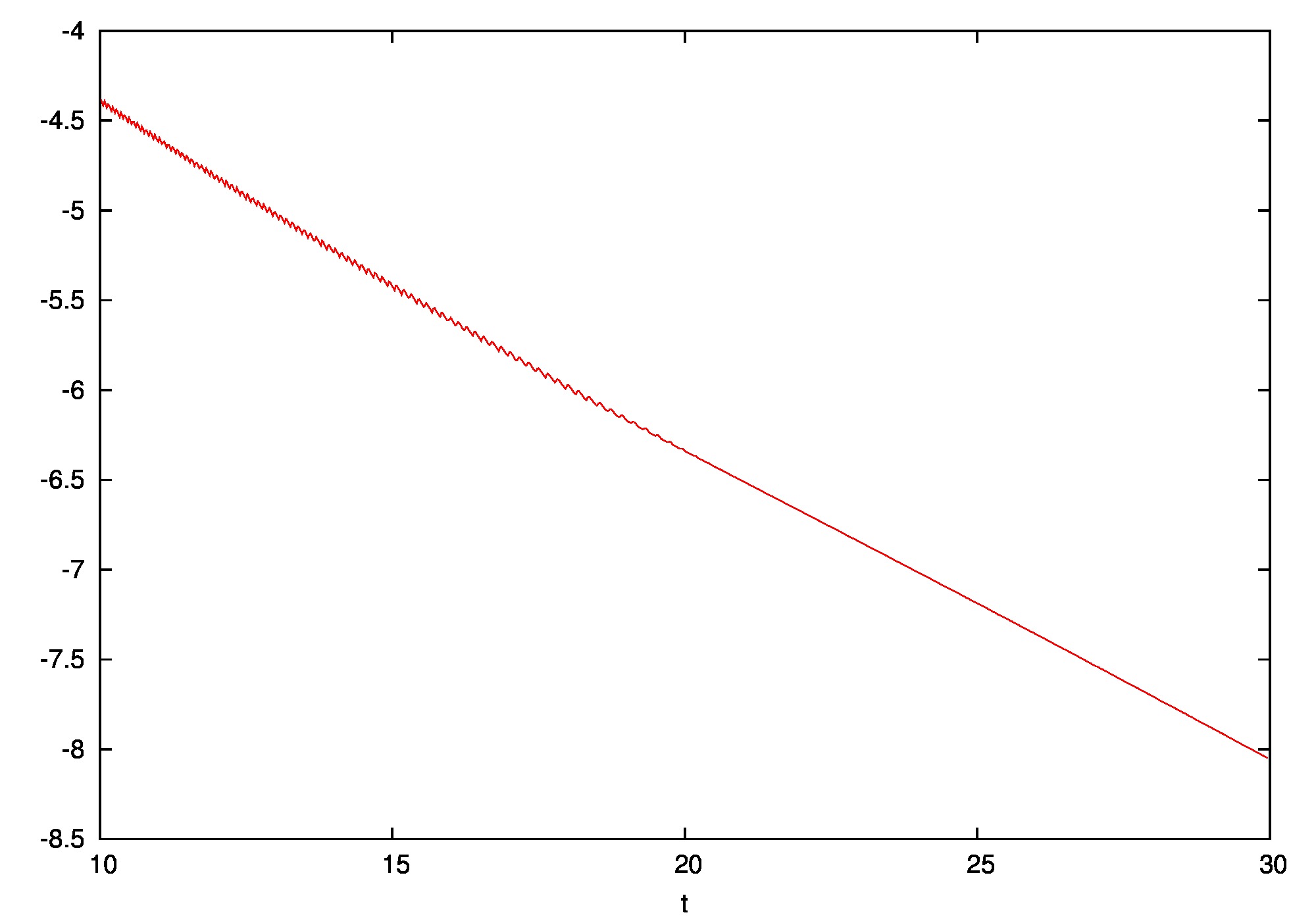}
\end{tabular}
\end{center}
\caption{Long-time convergence \eqref{eq:p_CV_correction} in the steady numerical frame.
To the left: the corrected error $e_{corr}(t)=\|\partial_tp(t,.)-\partial_xp(t,.)dX^*/dt(t)\|_\infty$ (in red) decays to zero, whereas the uncorrected $\|\partial_tp(t,.)\|_\infty$ (black) does not. 
To the right: plotting $\log e_{corr}(t)$ as a function of time indicates exponential convergence in the frame $x+X^*(t)$.
The parameters are $m=0.1$, $\alpha(y)=\alpha_2(y)$, $c=0.4$, and $\rd x=\rd y =5.e^{-3}$.}
\label{fig:errorLinfty_correction}
\end{figure}
This explains the previous residual error $\partial_tp\nrightarrow 0$ (see again Figure~\ref{fig:erreur_L2}), but also implies that we should have in the steady computational $x$-frame
\begin{equation}
\partial_tp-\partial_x p  \frac{dX^*}{dt}\underset{t\rightarrow +\infty}{\longrightarrow}0.
\label{eq:p_CV_correction}
\end{equation}
This convergence in the slowly drifting frame $x+X^*(t)$ can be checked numerically using \eqref{eq:def_p_tilda(t)}-\eqref{eq:Xdot_heuristc_copmutation} to compute the drift $dX^*/dt$: as shown in figure \ref{fig:errorLinfty_correction}, the convergence \eqref{eq:p_CV_correction} seems to be exponential and confirms our ansatz \eqref{eq:ansatz_p_shift}, and the scheme should therefore correctly approximate the stationary wave profiles in the long-time regime.
% %
%%%%%%%%%%%%%%%%%%%%%%%%%%%%%%%%%%%%%%%%%%%%%%%%%%%%%%%%%%%%%%%%%%%%%%%%%%%%%%%%%%%%%%%%%%%%%%%%%%%%%%%%%%%%%%
%
\section{Free boundaries and nondegeneracy}
\label{section:nondegeneracy}
As stated in Theorem~\ref{theo:main}, the free boundary $\Gamma=\partial\{p>0\}$ for the stationary wave profile can be parametrized in the privileged direction of propagation as
$$
p(x,y)>0\Leftrightarrow x>\mathcal I(y),
$$
where $\mathcal I(y)$ is a periodic upper semi-continuous function.
In most free boundary problems, one expects a gradient discontinuity across the interface, and the free boundary is said to be nondegenerate if the solution has a nontrivial growth across the interface.
This can be characterized by $\left|\nabla p_{}\right|_{\Gamma^+}\neq 0$, where $\left|\nabla p_{}\right|_{\Gamma^+}$ stands for the inner limit from the ``hot side'',
$$
\left|\nabla p_{}\right|_{\Gamma^+}:=\displaystyle{\lim_{(x,y)\overset{D_+}{\rightarrow}\Gamma}}|\nabla p|
$$
whenever this quantity makes sense.
At least for the pure PME $\partial_tp-mp\D p=|\nabla p^2|$, the regularity and propagation properties of the free boundary are strongly related to this nondegeneracy.
To see this one can formally evaluate the equation at the free boundary where $p$ vanishes, and get $\partial_t p=|\nabla p|^2$ along $\Gamma_t=\partial\{p(t,.)>0\}$ in some suitable (viscosity) sense.
This differential equation tells us that the free boundary moves in the outward normal direction with speed $-\left.\nabla p_{}\right|_{\Gamma^+_t}$ (the hot support $\{p>0\}$ invades the cold $\{p =0\}$, which is a remainder of the diffusive nature of the model at the interface), thus enlightening the role of the nondegeneracy.
In \cite{CaffWol-C1alpha} it is proved that, if the initial free boundary is nondegenerate at time $t=0$, then it starts to move immediately, never stops afterward, and remains nondegenerate.
The regularity and growth/nondegeneracy are also closely related to parabolic Harnack inequalities and monotonicity properties, see e.g. \cite{dahlberg1984non,DK07} and \cite{alt1984variational,caffarelli2005geometric}.
\begin{figure}[!ht]
\begin{center}
\begin{tabular}{cc}
\includegraphics[width=7cm]{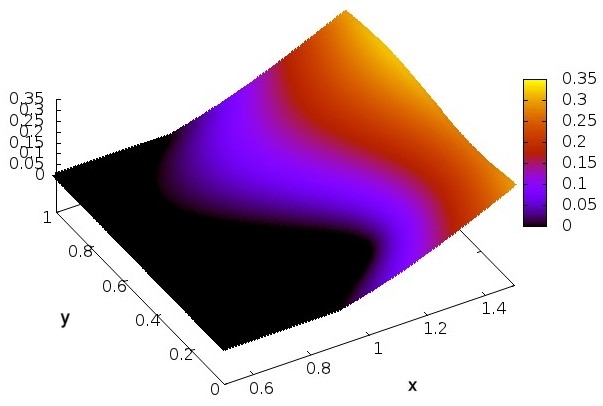} &
\includegraphics[width=7cm]{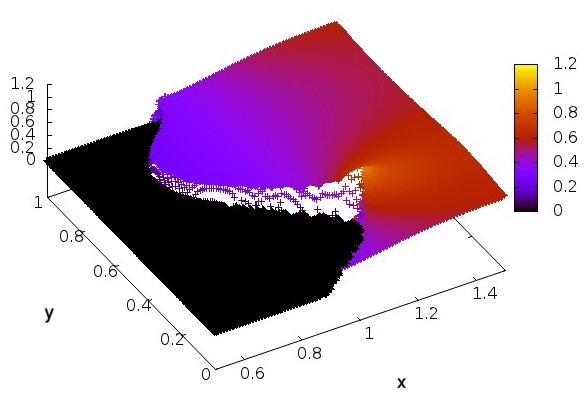}\\
\end{tabular}
\end{center}
\caption{Discontinuity of $\partial_x p$ across the free boundary.
Plots of $p(x,y)$ (left) and $\partial_x p(x,y)$ (right), with parameters $m=1.1$, $\alpha(y)=\alpha_1(y)$, $c=0.6$, $X_{\max}=10$, and $\rd x=\rd y= 5.e^{-3}$.
The time $t=30$ is large enough so that the Cauchy solution has converged to the stationary wave profile.}
\label{fig:px_discontinuity}
\end{figure}
\par
In the present case $\alpha(y)\not\equiv 0$ we expect a similar scenario: since we start with an initial datum $p^0(x,y)=c[x-\tau]^+$ whose free boundary is nondegenerate $\partial_x p^0|_{x=\tau^+}=c> 0$, the free boundary should stay nondegenerate as time evolves (although we do not claim here to prove this highly non-trivial statement).
This persisting gradient discontinuity across the interface should therefore be well adapted to detect the (moving) free boundary.
Since we are interested in traveling waves, the direction of propagation $x$ naturally plays an important role.
The numerical computations indeed always exhibit a jump of $\partial_x p$ across the free boundary, as illustrated in Figure~\ref{fig:px_discontinuity}, and therefore, supported by this numerical evidence, we always \emph{assume} that all the free boundaries remain nondegenerate.
Thus $\partial_x p$ will be a relevant quantity to monitor, and in fact the nondegeneracy $\left.\partial_x p\right|_{\Gamma}\geq c_0>0$ will be the key ingredient to prove the Lipschitz regularity in section~\ref{section:regularity_corners} (see Theorem~\ref{theo:regularity} later on).
The discontinuity of $\partial_x p$ moreover leads to a (positive) Dirac singularity $\partial^2_{xx}p=+\infty$ at $\Gamma_t$, and this blowup is very easy to track numerically in order to detect the free boundary (recall from Theorem~\ref{theo:main} that the solution is $\mathcal C^{\infty}$ in $\{p>0\}$, hence singularities of $\partial^2_{xx}p$ can only occur at the interface).
Once the free boundary is detected we can compute numerically $\partial_x p$ ``at'' $\Gamma$ and check the nondegeneracy.
More precisely, we use

\begin{algo}[Free Boundary detection by $\partial^2_{xx}p$ singularity and computation of $\left.\partial_x p\right|_{\Gamma^+}$]
\label{algo:FB_detection_pxx}
Choose some integer parameter $s>0$, and for each $j\in\llbracket 1,N_y\rrbracket$:
\begin{enumerate}
\item for $i\in \llbracket 2,N_x-1 \rrbracket$, compute $\D_{xx} P_{i,j}=\cfrac{P_{i+1,j}+P_{i-1,j}-2P_{i,j}}{\rd x^2}$
\item find the maximum of $\D_{xx}P_{i,j}$ over $i\in\llbracket 2,N_x-1\rrbracket$, and denote by $I(j)$ it's location
\item the position of the free boundary is given by $\mathcal I(y_j)\approx x_{I(j)}$
\item 
compute $\partial_x p$ at the free boundary as $\left.\partial_x p\right|_{\Gamma^+}(y_j)\approx\cfrac{P_{i_0+s+1,j}-P_{i_0+s,j}}{\rd x}$
\end{enumerate}
\end{algo}
Note that this detection procedure only sweeps in space (in the $x$ direction), and can therefore be used dynamically for any fixed time to detect moving free boundaries $\Gamma_t=\partial\{p(t,.)>0\}$ for the Cauchy problem \eqref{eq:PME-S_pressure_parabolic_waveframe}.
Due to the numerical diffusion smoothing the gradient discontinuity, $\partial_x p$ actually jumps across a small $\mathcal O(\rd x)$ numerical boundary layer where the finite difference approximations may become inaccurate.
This is why we choose to step $s$ meshes away in the $x>0$ direction when computing what should be $\partial_x p$ at the free boundary (step 4).
In our experiments the choice $s=5$ gave satisfactory results even for high resolutions $\rd x=\rd y=1.e^{-3}$.
Note also that the approximation $\left.\partial_x p\right|_{\Gamma^+}\approx\cfrac{P_{i_0+s+1,j}-P_{i_0+s,j}}{\rd x}$ in step 4 is a forward difference: the relevant information to compute $\left.\partial_x p\right|_{\Gamma^+}$ ``at the free boundary'' should come indeed from the hot side $D_+=\{p>0\}$, here to the $x>\mathcal I(y)$ side of the free boundary as in Theorem \ref{theo:main}. 

A typical result obtained with Algorithm~\ref{algo:FB_detection_pxx} is illustrated in Figure \ref{fig:algo_FB_detection}: clearly $\partial_x p$ jumps across the interface and stays bounded away from zero in $\{p>0\}$, the solution grows linearly in $x$ across $\Gamma$, and the interface is thus nondegenerate in the direction of propagation.
Note the apparent boundary layer and numerical diffusion in the bottom left image, where a few meshpoints are visible inside the numerical boundary layer.
%figure comment
\begin{figure}[h!]
\begin{center}
\begin{tabular}{cc}
\includegraphics[width=7cm]{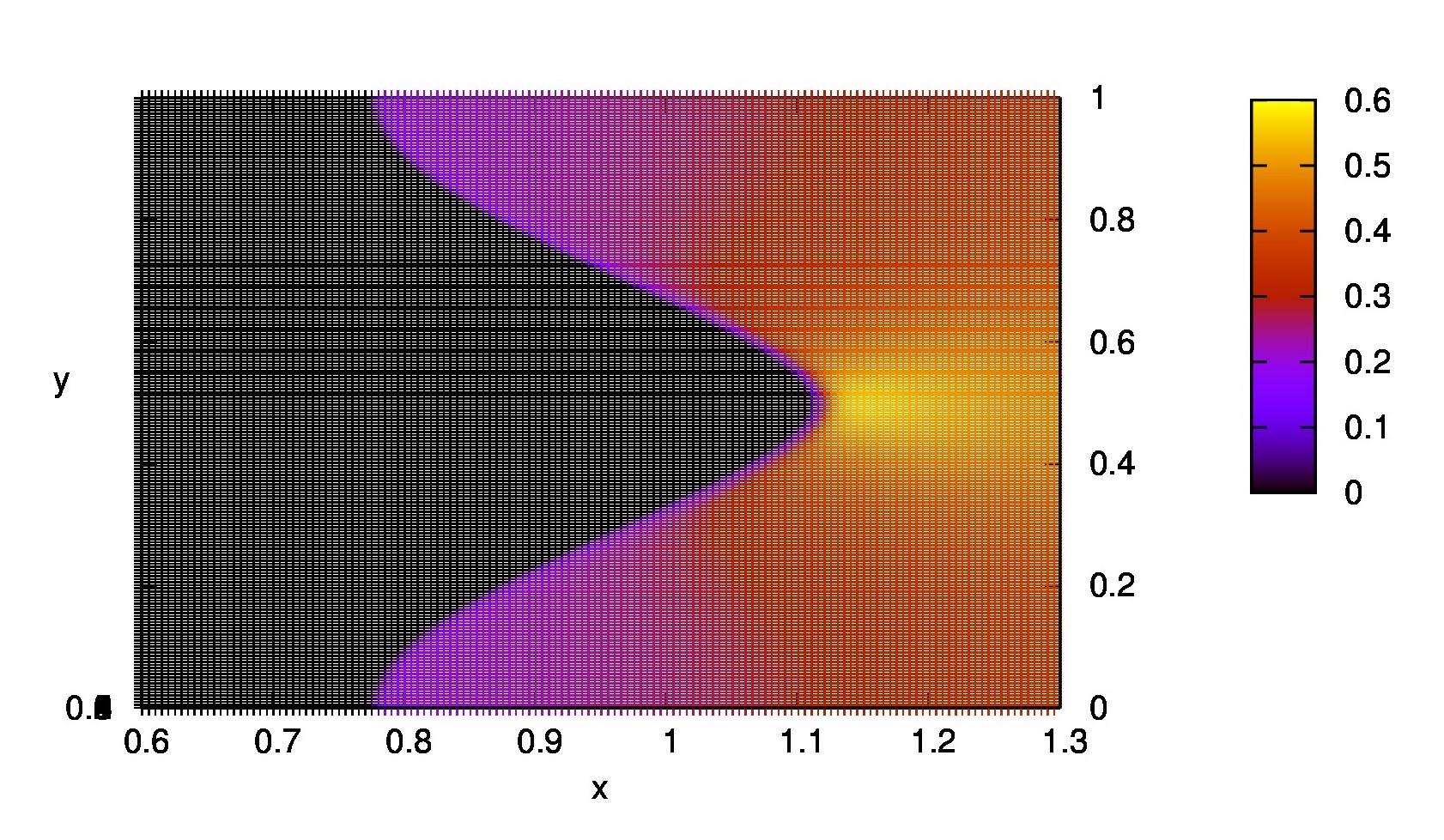} &
\includegraphics[width=7cm]{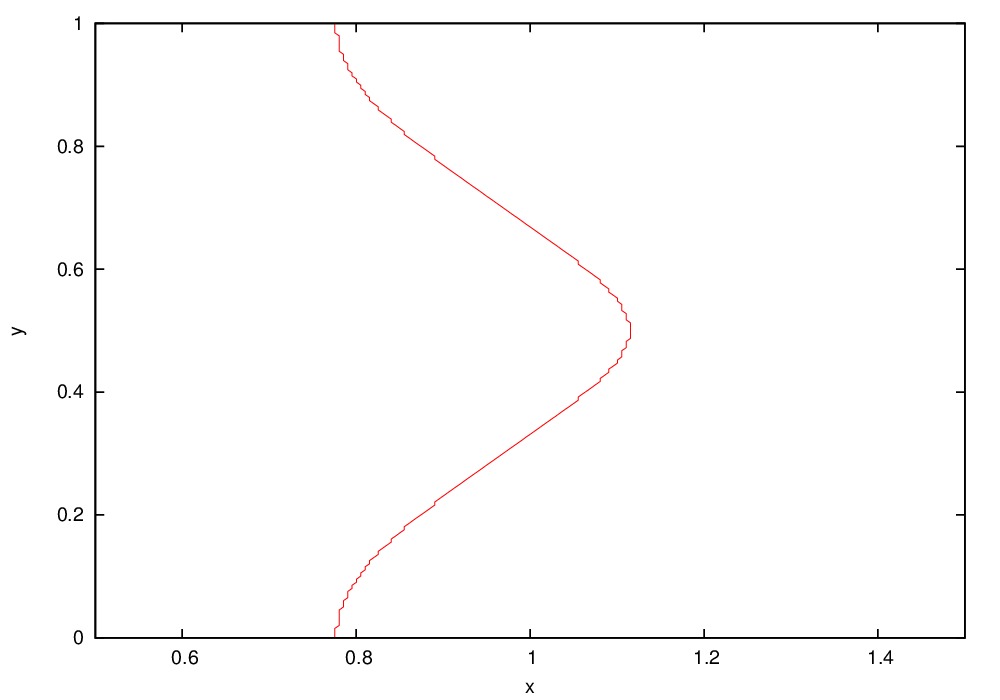}\\
 \includegraphics[width=7cm]{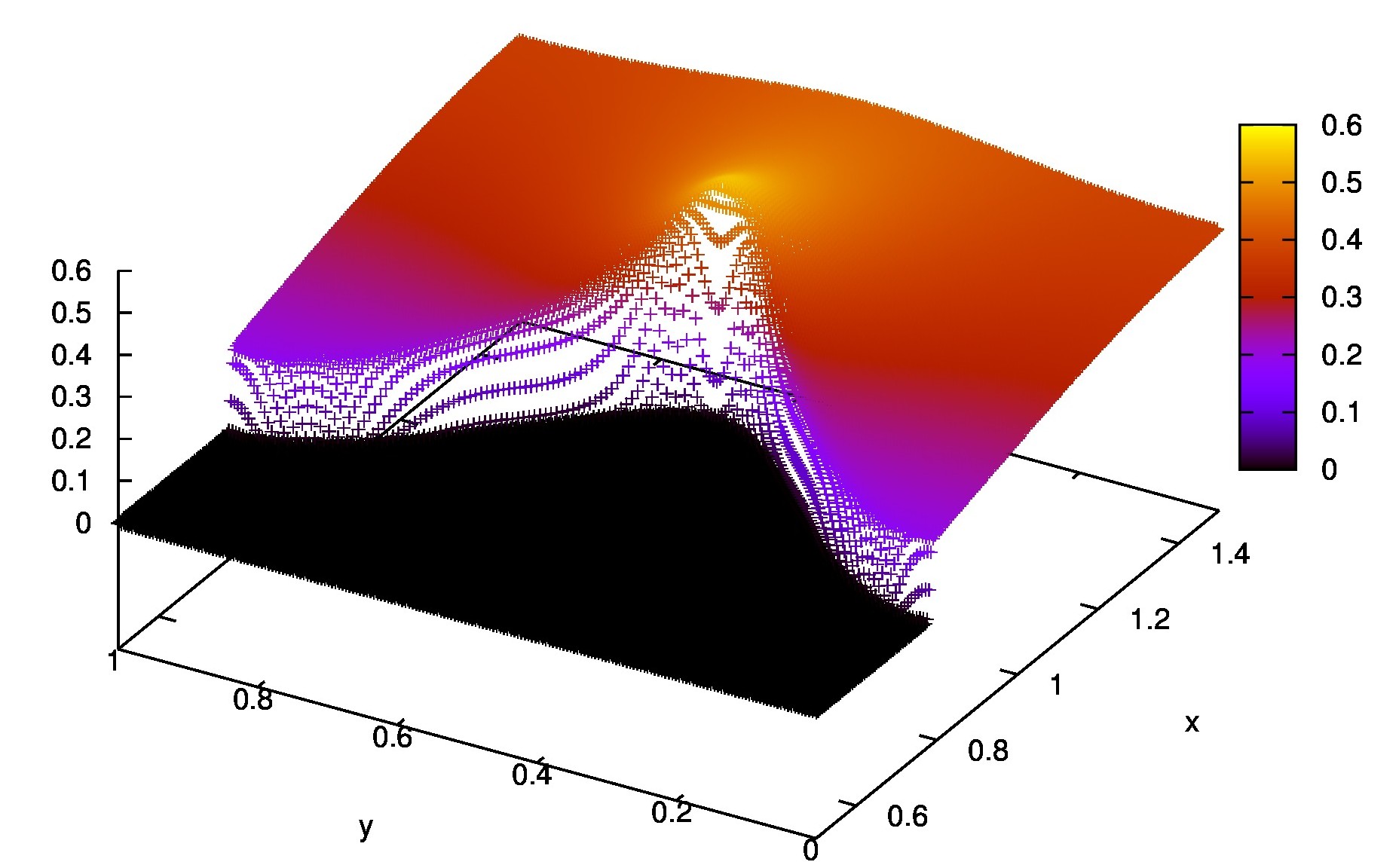} &
 \includegraphics[width=7cm]{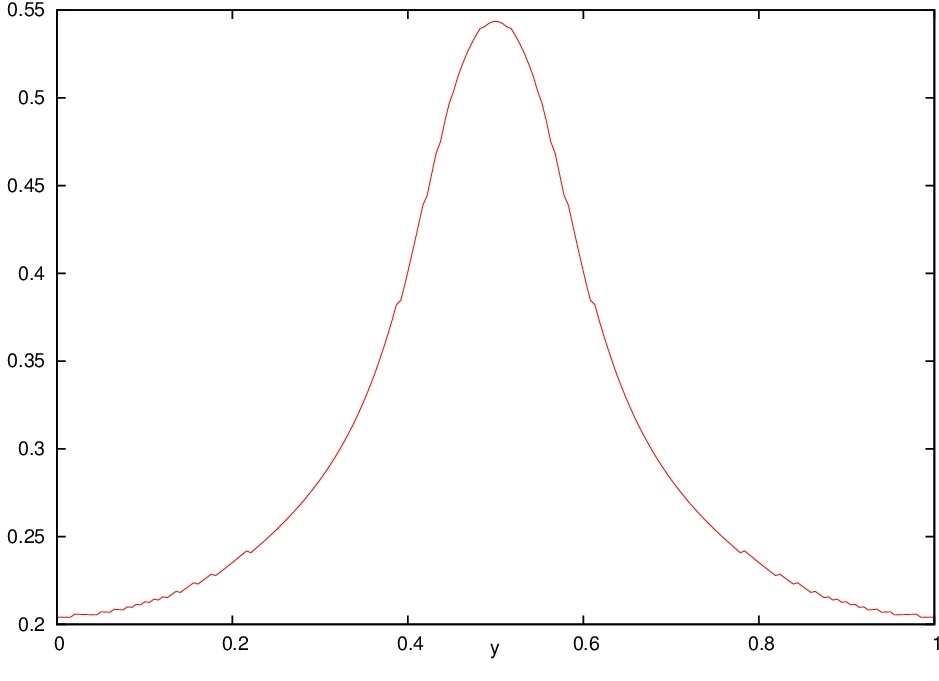}\\
\end{tabular}
\end{center}
\caption{Example of a numerical computation with Algorithm~\ref{algo:FB_detection_pxx}.
View of $\partial_x p$ from different angles (top and bottom left), detection of the free boundary $x=\mathcal I(y)$ (top right), and computation of $\left.\partial_x p\right|_{\Gamma^+}$ as a function of $y$ (bottom right). The parameters are $m=1.1$, $\alpha(y)=\alpha_2(y)$, $c=0.4$, $X_{\max}=10$, $\rd x=\rd y=5.e^{-3}$.
The time $t=30$ is large enough so that the Cauchy solution has converged to the stationary wave profile.}
\label{fig:algo_FB_detection}
\end{figure}

%
%%%%%%%%%%%%%%%%%%%%%%%%%%%%%%%%%%%%%%%%%%%%%%%%%%%%%%%%%%%%%%%%%%%%%%%%%%%%%%%%%%%%%%%%%%%%%%%%%%%%%%%%%%%%%%
%

\section{Interface regularity and corners}
\label{section:regularity_corners}
In the rest of the paper we consider the general $d$ dimensional case $(x,y)\in\R\times\mathbb{T}^{d-1}$ for the analysis, but still restrict to $d=2$ for the simulations.
%
%%%%%%%%%%%%%%%%%%%%%%%%%%%%%%%%%%%%%%%%%%%%%%%%%%%%%%%%%%%%%%%%%%%%%%%%%%%%%%
%
\subsection{A regularity result}
\label{subsection:regularity_result}
Let us recall from Theorem~\ref{theo:main} that $p\in\mathcal{C}^{\infty}(D_+)$ and $\partial_x p>0$ in  $D_+=\{p>0\}$.
Hence for any $\eps>0$ the levelset $\Gamma_{\eps}=\{p=\eps\}$ can be parametrized by the Implicit Functions Theorem as a graph
$$
(x,y)\in\Gamma_{\eps}\quad\Leftrightarrow\quad p(x,y)=\eps\quad\Leftrightarrow\quad x=X_{\eps}(y), \quad y\in\mathbb{T}^{d-1}
$$
for some function $X_\eps:\mathbb T^{d-1}\to \R$ satisfying in particular
\begin{equation}
\nabla_yX_{\eps}(y)=-\cfrac{1}{\partial_x p\left(X_{\eps}(y),y\right)}\nabla_yp\left(X_{\eps}(y),y\right).
\label{eq:Xeps'=}
\end{equation}
Differentiating again w.r.t $y$, using \eqref{eq:PME-S_pressure_stationary}, and dividing by $\partial_x p\left(X_{\eps}(y),y\right)>0$ leads to
\begin{multline}
-\eps\frac{m}{\partial_x p} \D_{y}X_{\eps}
+\underbrace{\left|\nabla_y X_{\eps}\right|^2}_{:=H(\nabla_yX_\eps)}\\
+\eps\underbrace{\frac{m}{\partial_x p}\left[\partial^2_{xx}p\left(1-\left|\nabla_y X_{\eps}\right|^2\right)-2\partial^2_{xy}p\cdot\nabla_y X_{\eps}\right]}_{:=H_{1,\eps}(y,\nabla_yX_\eps)}=\underbrace{\left(\cfrac{c+\alpha}{\partial_x p}-1\right)}_{:=g_\eps(y)},
\label{eq:HJ_eps}
\end{multline}
where the terms $\partial_x p,\partial_{xx}^2p\in \R$ and $\partial^2_{xy}p\in \R^{d-1}$ are implicitly evaluated at the $\eps$-levelset $x=X_\eps(y)$ and considered as functions of $y\in\T$ only with a slight abuse of notations.
Because $\Gamma=\partial\{p>0\}$ is in some sense the $0$-levelset closest to $\{p>0\}$, it seems natural that we should recover the free boundary parametrization $\{x=\mathcal I(y)\}$ as the limit of the $\eps$-parametrization $\{x=X_{\eps}(y)\}$ along the levelset descent $\eps\searrow 0$.
This is exactly the strategy of proof of the following regularity result:
\begin{theo}
When $\eps\rightarrow 0^+$ assume that
\begin{equation}
\eps\left(\left|\partial^2_{xx}p\right|_{\Gamma_{\eps}}(y)+\left|\partial^2_{xy}p\right|_{\Gamma_{\eps}}(y)\right) \overset{L^{\infty}(\mathbb T^{d-1})}{\longrightarrow}0,
\label{hyp:H1}
\tag{$\mathcal{H}1$}
\end{equation}
\begin{equation}
f_{\eps}(y):=\left. \partial_x p\right|_{\Gamma_{\eps}}(y)\overset{L^{\infty}(\mathbb T^{d-1})}{\longrightarrow} f(y)>0
\label{hyp:H2}
\tag{$\mathcal{H}2$}
\end{equation}
for some limit $f(y)\geq C>0$.
Then the interface parametrization $\mathcal I(y)$ is a periodic viscosity solution of the Hamilton-Jacobi equation
\begin{equation}
\left|\nabla_yI\right|^2=g\quad \text{in }\mathbb T^{d-1},
\qquad g(y):=\left(\cfrac{c+\alpha(y)}{f(y)}-1\right).
\label{eq:HJ_0}
\tag{HJ}
\end{equation}
Moreover $\mathcal I(.)$ is globally Lipschitz and semi-concave, and the free boundary $\Gamma=\partial\{p>0\}$ coincides with the Lipschitz graph $\{x=\mathcal I(y)\}$.
In dimension $d=2$ the function $\mathcal I(y)$ is everywhere left and right differentiable on $\T$.
\label{theo:regularity}
\end{theo}
Note that Theorem~\ref{theo:regularity} holds in any dimension $d\geq 2$, and that the Hamilton-Jacobi equation \eqref{eq:HJ_0} is not equivalent to $-|\nabla_y \mathcal I |^2=-g$ in the viscosity sense (this will be important later on in section~\ref{subsec:exist_corners}).
Let us first comment on our hypotheses \ref{hyp:H1}\ref{hyp:H2}, which will be validated numerically in the next section~\ref{subsection:num_validation_H1H2}. The first assumption \ref{hyp:H1} is rather technical and is only needed to retrieve the Hamilton-Jacobi equation, but seems reasonable compared to the usual PME scenario.
Indeed the PME planar wave $p_c(x,y)=c[x]^+$ is linear $\partial^2_{xy} p\equiv 0$ in $\{p_c>0\}$ and trivially satisfies \eqref{hyp:H1}, and for more general situations this hypothesis is consistent with the celebrated Aronson-B\'enilan semiconvexity estimate $\D p(t,.)\geq -\lambda/t$ in the steady frame \cite{AronsonBenilan-regularite,Vazq-PME} (giving in particular $\D p\geq 0$ in the wave frame, taking $t\to\infty$ for the stationary wave profile).
Assumption \ref{hyp:H2} is a more fundamental nondegeneracy condition. It can be viewed as an approximation to the strong nondegeneracy $\left.\partial_x p\right|_{\Gamma_+}>0$ along descending levelsets $\eps\searrow 0$, which is consistent with and supported by the numerical results in section~\ref{section:nondegeneracy}.
A closer look at the proof below will reveal that, regardless of \eqref{hyp:H1}, this second condition will imply that the free boundary is really the Lipschitz graph $\Gamma=\{x=\mathcal I(y)\}$, which so far was not guaranteed by Theorem \ref{theo:main}.
It is worth stressing that in \cite{MonsaingonNovikovRoquejoffre} we were not able to prove analytically any of the strong statements \eqref{hyp:H1}\eqref{hyp:H2}, and this is precisely why we use numerical computations in this paper as an investigation tool.
\begin{proof}
Considering $\left.\partial_x p\right|_{\Gamma_{\eps}},\left.\partial^2_{xx}p\right|_{\Gamma_{\eps}},\left.\partial^2_{xy}p\right|_{\Gamma_{\eps}}$ as known functions of $y$, \eqref{eq:HJ_eps} can be recast as
$$
-\eps \frac{m}{f_{\eps}(y)}\D_{y}X_{\eps}+H\left(\nabla_yX_{\eps}\right)+\eps H_{1,\eps}\left(y,\nabla_yX_{\eps}\right)=g_{\eps}(y)
$$
with $f_{\eps}(y):=\left.\partial_xp\right|_{\Gamma_{\eps}}(y)$ and Hamiltonians $H(\zeta)$, $H_{1,\eps}(y,\zeta)$ and $g_\eps(y)$ as in \eqref{eq:HJ_eps}.
Let us recall that \ref{hyp:H2} implies that $f_{\eps}(y)\sim f(y)\geq C>0$ uniformly in $y$ for $\eps$ small, and this PDE for $X_\eps$ is therefore uniformly elliptic for fixed $\eps>0$.

By \eqref{eq:Xeps'=} with $\|\partial_y p\|_\infty\leq C$ from Theorem~\ref{theo:main}(ii), we see that \eqref{hyp:H2} gives Lipschitz bounds on $X_{\eps}(.)$ uniformly in $\eps$.
We can therefore assume that
$$
X_{\eps}(.)\rightarrow X_0(.)\qquad 
\text{uniformly in } \T
$$
at least for some discrete subsequence, and the limit $X_0$ is also Lipschitz.
Hypothesis \ref{hyp:H2} implies $\partial_x p\geq C>0$ in some right neighborhood of the interface, and this strong monotonicity condition together with the characterization \eqref{eq:char_I(y)} of $\mathcal I(y)$ immediately imply that the limit is actually
$$
X_0(y)=\displaystyle{\lim_{\eps\rightarrow 0}}X_{\eps}(y)=\inf\left\{x\in\R:\, p(x,y)>0\right\}=\mathcal I(y).
$$
Because $\mathcal I(.)$ is now Lipschitz we can conclude by continuity that the free boundary $\Gamma=\partial\{x>\mathcal I(y)\}$ is really the graph $\{x=\mathcal I(y)\}$.

In order to establish \eqref{eq:HJ_0}, observe that \ref{hyp:H1}-\ref{hyp:H2} imply convergence $-m\eps/ f_{\eps}(.)\rightarrow 0$ uniformly on $\T$, and
$$
\eps H_{1,\eps}(y,\zeta)  \rightarrow  0,
\qquad
g_\eps(y)\to g(y)=\left(\cfrac{c+\alpha(y)}{f(y)}-1\right)
$$
locally uniformly in $(y,\zeta)\in\mathbb{T}^{d-1}\times\R^{d-1}$ when $\eps\rightarrow 0^+$.
By standard stability theorems \cite{Barles,CrandallIshiiLions-userguide} the uniform limit $\mathcal I(y)=X_0(y)=\displaystyle{\lim_{\eps\rightarrow 0}}\; X_{\eps}(y)$ is a viscosity solution of the limiting equation $H(\nabla_yX_0)=g$, which is exactly \eqref{eq:HJ_0}.
Note that this is exactly the classical construction of viscosity solutions by the vanishing viscosity method (up to the first order perturbation $\eps H_{1,\eps}(y,\zeta)\to 0$).
Since our periodic solution $X_0(y)=\mathcal I(y)$ is globally Lipschitz on the torus and the Hamiltonian $H(\zeta)=|\zeta|^2$ is coercive, the results in \cite{Roquejoffre-PropQualHJ} for stationary solutions apply and $\mathcal I(y)$ is semi-concave in any dimension $d\geq 2$.
The stronger everywhere left- and right-differentiability in dimension $d=2$ is finally a direct consequence of \cite[Theorem 1]{JensenSouganidis-regularity}, or alternatively follows by semiconcavity.
\end{proof}

%
%%%%%%%%%%%%%%%%%%%%%%%%%%%%%%%%%%%%%%%%%%%%%%%%%%%%%%%%%%%%%%%%%%%%%%%%%%%%%%%%%%%%%%%%%%%%%%%%%%%%
%
\subsection{Numerical validation of hypotheses \ref{hyp:H1}-\ref{hyp:H2}}
\label{subsection:num_validation_H1H2}
In Theorem~\ref{theo:regularity} we \emph{assumed} the regularity \eqref{hyp:H1} as well as the nondegeneracy condition \eqref{hyp:H2} to get the Hamilton-Jacobi equation \eqref{eq:HJ_0} and the Lipschitz regularity of the free boundary, and we need to validate numerically those strong assumptions.
Since $\partial_x p>0$ in $\{p>0\}$ it is easy to detect the levelsets by sweeping the numerical solution in the $x<0$ direction.
More precisely, for fixed $\eps>0$ we approximate the $\eps$-levelset $\Gamma_{\eps}=\{x=X_{\eps}(y)\}$ as
\begin{equation}
X_{\eps}(y_j)\approx x_{I_{\eps}(j)},\qquad I_{\eps}(j)=\min\{i\in \llbracket 1 ,N_x\rrbracket:\, P_{i,j}\geq \eps\}.
\label{eq:approx_eps_levelsets}
\end{equation}
We then compute finite difference approximations to $\partial^2_{xx}p$ and $\partial^2_{xy}p$ at the $\eps$-levelset as
$$
\left.\partial^2_{xx}p\right|_{\Gamma_{\eps}}(y_j)		\approx 	\Delta^2_{xx}P_{I_\eps(j),j}	=	\frac{P_{I_\eps(j)-1,j}-2P_{I_\eps(j),j}+P_{I_\eps(j)+1,j}}{\rd x^2}
$$
$$
\left.\partial^2_{xy}p\right|_{\Gamma_{\eps}}(y_j)	\approx	 \Delta_x(\Delta_y P_{I_{\eps}(j),j})=\frac{
\left(P_{I_\eps(j)+1,j+1}-P_{I_\eps(j)+1,j-1}\right)-\left(P_{I_\eps(j)-1,j+1}-P_{I_\eps(j)-1,j-1}\right)
}{4\rd x\, \rd y},
$$
and the slope is simply
$$
\left.\partial_{x}p\right|_{\Gamma_{\eps}}(y_j)	\approx \Delta_x P_{I_{\eps}(j),j}=\frac{P_{I_\eps(j)+1,j}-P_{I_\eps(j)-1,j}}{2\rd x}.
$$
The result is shown in Figures \ref{fig:H1}-\ref{fig:H2}: $\left.\eps\partial_{xx}^2p\right|_{\Gamma_\eps}$ and $\left.\eps\partial_{xy}^2p\right|_{\Gamma_\eps}$ appear to converge uniformly to zero as in \eqref{hyp:H1}, and $\left.\partial_x p\right|_{\Gamma_{\eps}}$ stays bounded away from zero as in \eqref{hyp:H2}.
%
%figure comment
\begin{figure}[h!]
\begin{center}
\begin{tabular}{cc}
\includegraphics[width=7cm]{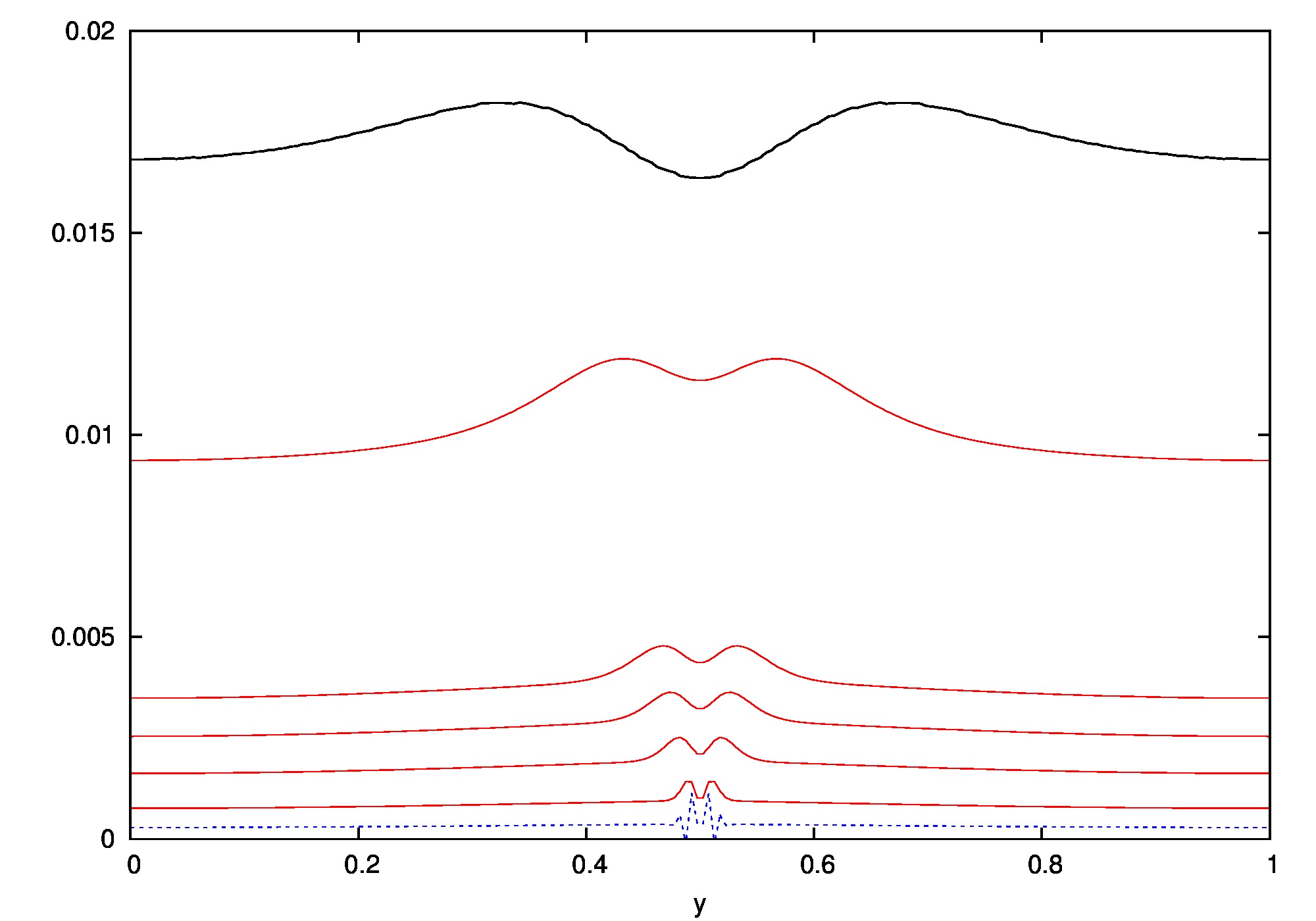} &
\includegraphics[width=7cm]{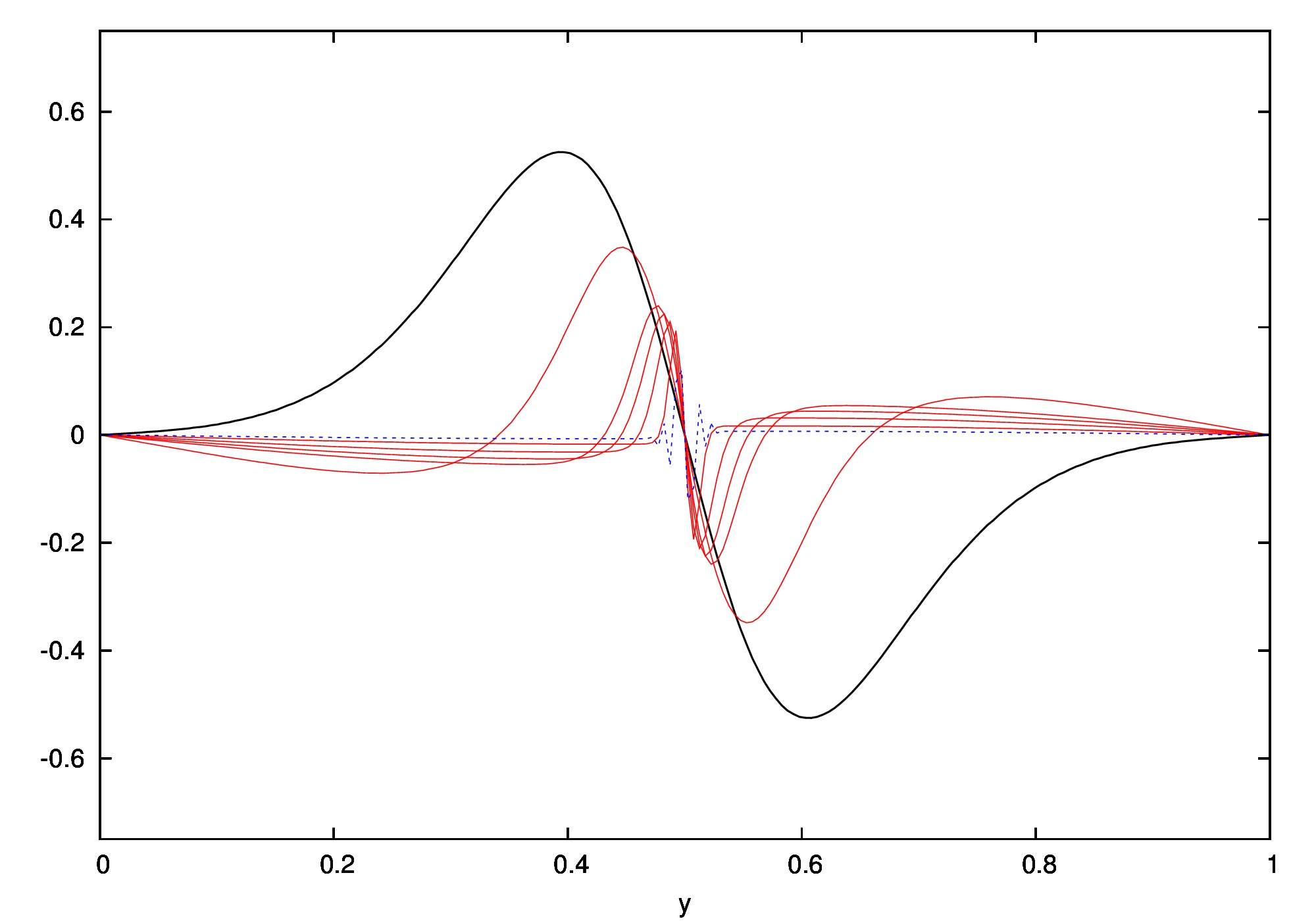}
\end{tabular}
\end{center}
\caption{Numerical validation of hypothesis \ref{hyp:H1}. $\left.\eps\partial^2_{xx}p\right|_{\Gamma_\eps}$ (left) and $\left.\eps\partial^2_{xy}p\right|_{\Gamma_\eps}$ (right) as functions of $y$ at the $\eps$-levelset, plotted for several values of $\eps$ between $5\cdot 10^{-1}$ (black) and $1\cdot 10^{-2}$ (blue).
The parameters are $\alpha(y)=\alpha_2(y)$, $c=0.4$, $m=0.1$, $X_{\max}=10$, $\rd x=\rd y= 5.e^{-3}$.
The time $t=30$ is large enough so that the Cauchy solution has converged to the stationary wave profile.}
\label{fig:H1}
\end{figure}
\begin{figure}[h!]
\begin{center}
\begin{tabular}{cc}
\includegraphics[width=7cm]{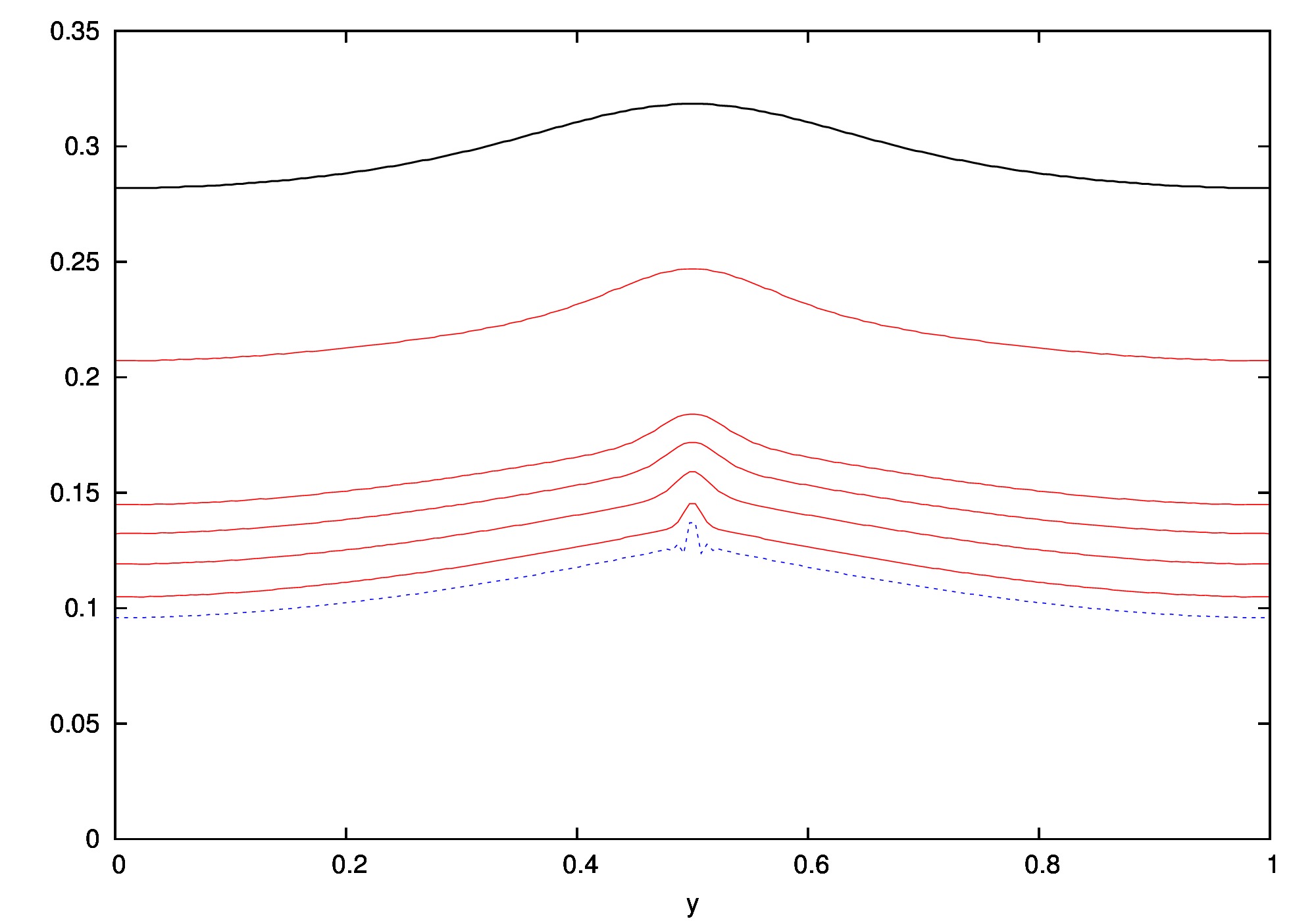} & \includegraphics[width=7cm]{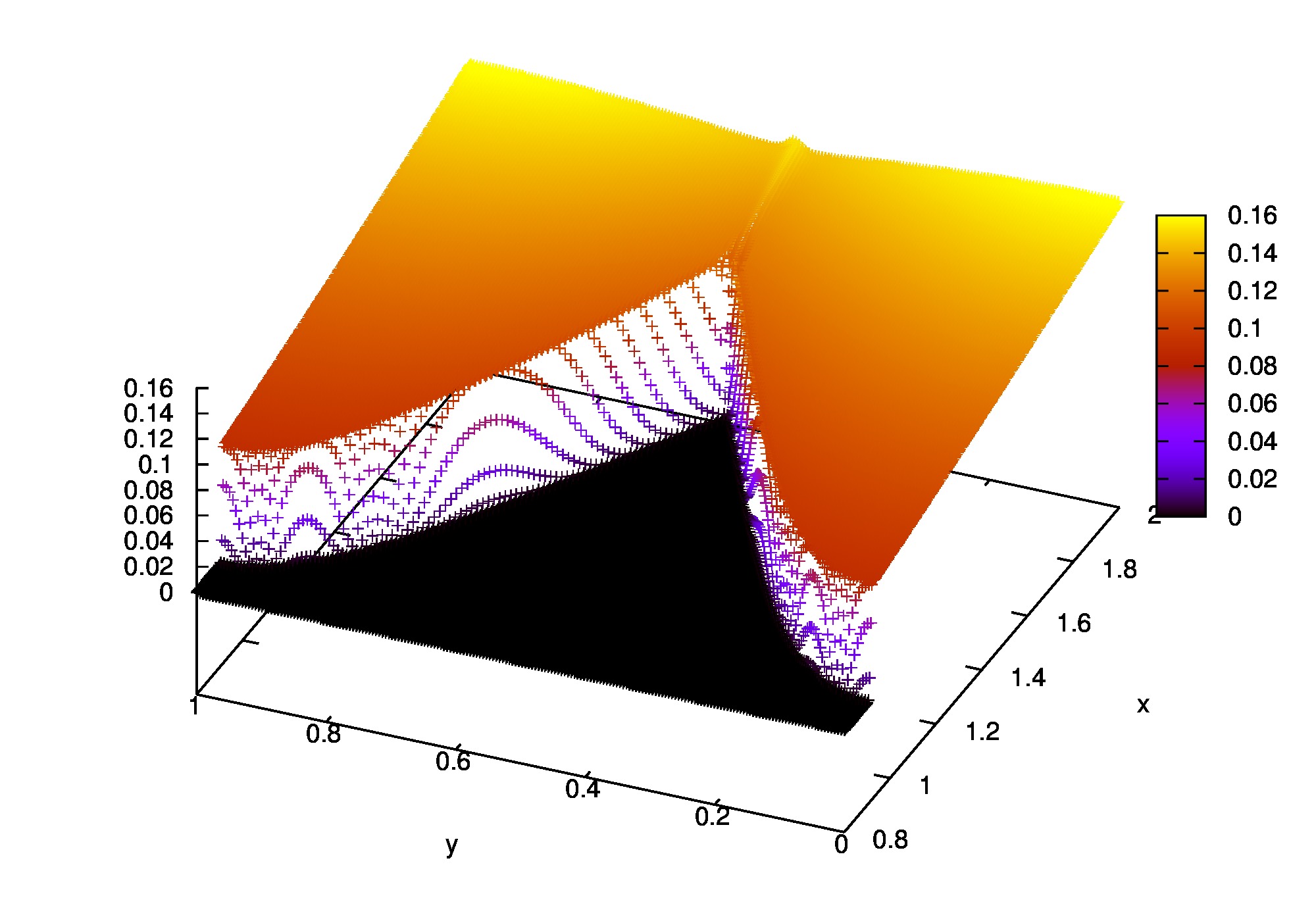}
\end{tabular}
\end{center}
\caption{Numerical validation of hypothesis \ref{hyp:H2}. To the left: $\left.\partial_{x}p\right|_{\Gamma_\eps}$ as a function of $y$ at the $\eps$-levelset, plotted for several values of $\eps$ between $5\cdot 10^{-1}$ (in black) and $1\cdot 10^{-2}$ (in blue).
To the right: view of $\partial_xp$ as a function of $(x,y)$ close to the free boundary. The parameters are as in Figure \ref{fig:H1}.}
\label{fig:H2}
\end{figure}
For practical reasons and since $p$ grows linearly in $x$ across the free boundary, $\eps>0$ cannot be chosen too small compared to the mesh size $\rd x,\rd y$ in order for the $\eps$-levelset to be consistent.
Indeed if $\eps$ is too small the numerical levelset $\Gamma_{\eps}=\{x=X_\eps(y)\}\leftrightarrow \{i=I_{\eps}(j)\}$ is eventually detected by \eqref{eq:approx_eps_levelsets} inside the $\mathcal O(\rd x)$ numerical boundary layer surrounding the actual free boundary, and the spatial derivatives become inaccurate.
This is apparent in Figures~\ref{fig:H1} and \ref{fig:H2} with $\rd x=\rd y=5.e^{-3}$: numerical oscillations start to develop for the smallest tested value $\eps=1.e^{-2}$.
%
%%%%%%%%%%%%%%%%%%%%%%%%%%%%%%%%%%%%%%%%%%%%%%%%%%%%%%%%%%%%%%%%%%%%%%%%%%%%%%%%%%%%%%%%%%%%%%%%%%%%
%
\subsection{Existence of corners}
\label{subsec:exist_corners}
According to Theorem \ref{theo:regularity}, the interface parametrization $\mathcal I(y)$ is semi-concave as a periodic solution to the Hamilton-Jacobi equation \eqref{eq:HJ_0}, of the form
$$
|\nabla_y\mathcal I|^2=g(y)
\qquad \text{in }\mathbb T^{d-1}.
$$
Roughly speaking, semi-concavity means that $\mathcal I$ has only smooth ($\mathcal{C}^1$) minimum points, but that corner-shaped maximum points are allowed.
This semiconcavity comes from the minus sign in the diffusive term $-\eps\frac{m}{\partial_x p} \D_{y}X_{\eps}$ for the vanishing viscosity approximation \eqref{eq:HJ_eps}, or equivalently from the fact that in the limit $\eps=0$ we are solving solving $+|\nabla_y \mathcal{I}|^2=g$ in the viscosity sense and not $-|\nabla_y \mathcal{I}|^2=-g$ (which in general is not equivalent, see e.g. \cite{CrandallIshiiLions-userguide,Barles}).

It is well known \cite{Barles-UniquenessHJ1,Barles,Roquejoffre-PropQualHJ} that the uniqueness and regularity of such periodic solutions strongly depend on the number of zeros of the right-hand side $g(y)$ in the torus.
A necessary condition for classical $\mathcal{C}^1$ solutions to exist is that $g(y)$ should vanish at least twice, in which case uniqueness fails.
If now $g(y)$ vanishes only once, any solution to the Hamilton-Jacobi equation has nonsmooth semiconcave maximum points (in agreement with Theorem~\ref{theo:regularity}), and the free boundary $\{x=\mathcal I(y)\}$ should accordingly develop corners pointing in the $x>0$ direction .
Indeed in this case $\nabla_y \mathcal{I}$ can vanish only once at a minimum point (the unique zero $g(y_0)=0$), but $I$ has at least one maximum point where the derivative fails to exist and the equation cannot be satisfied pointwise (but in the viscosity sense).
As a consequence we expect minimum points to be regular, whereas maximum points may be generically corner shaped.

Strangely enough, our simulations indicate that the interface has corners only for diffusion exponents $m\in(0,1)$.
For $m>1$ we could never observe corners, and the interfaces always looked like regular $\mathcal{C}^1$ graphs.
This is illustrated in Figure~\ref{fig:corners_not_corners}, where the same computations are compared for $m<1$ and $m>1$: corners clearly appear in the first case, whereas the free boundaries look smooth in the latter case.
Our computations also suggest that $m=1$ is a sharp threshold, at least in dimension $d=2$: corners systematically developed for all tested values up to $m<0.99$, and completely disappeared as early as $m\geq 1.01$.
We do not have any clue for this phenomenon and cannot explain the threshold so far.
\begin{figure}[h!]
\begin{center}
\begin{tabular}{cc}
\includegraphics[width=7cm]{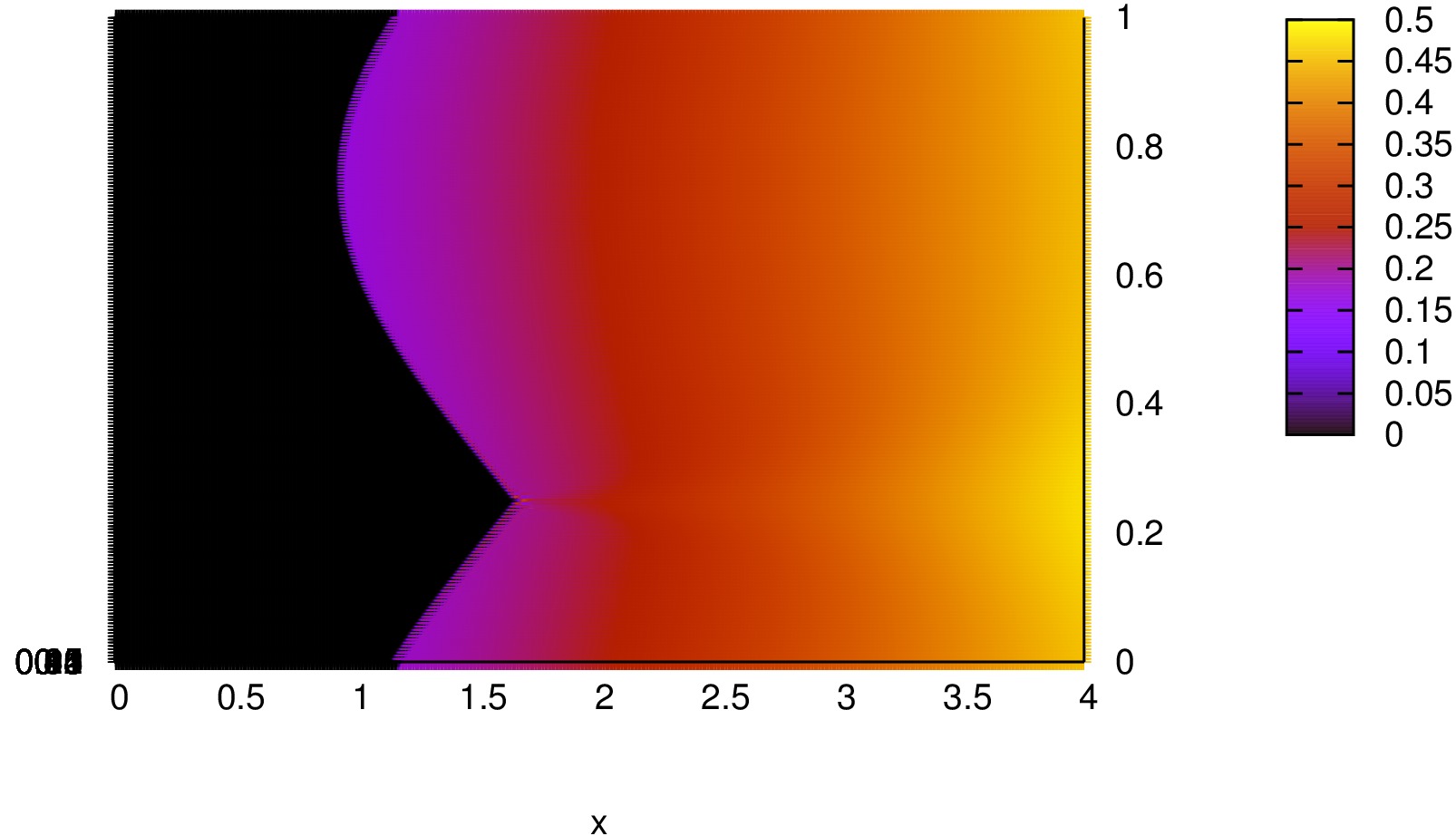} &
\includegraphics[width=7cm]{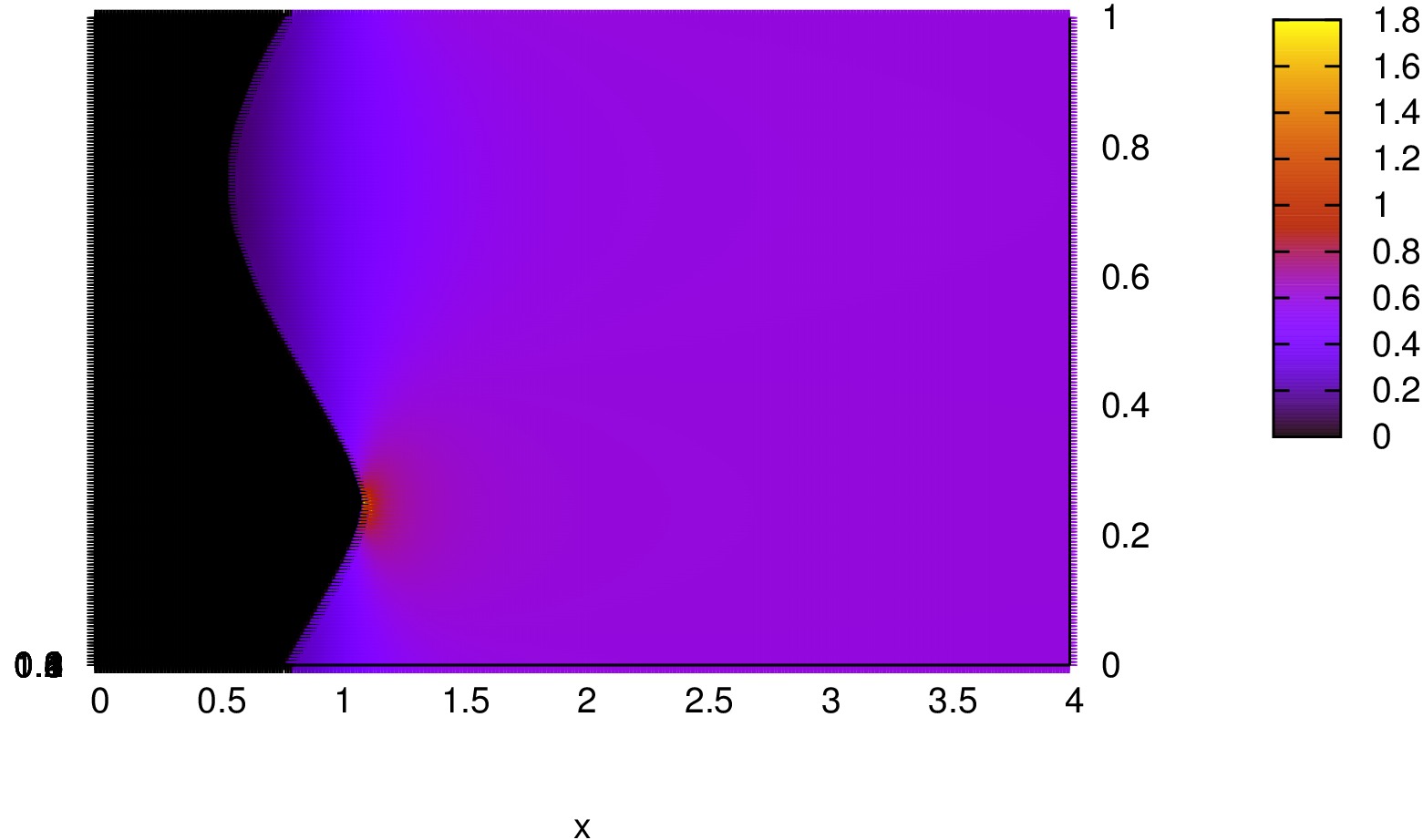}\\
\includegraphics[width=7cm]{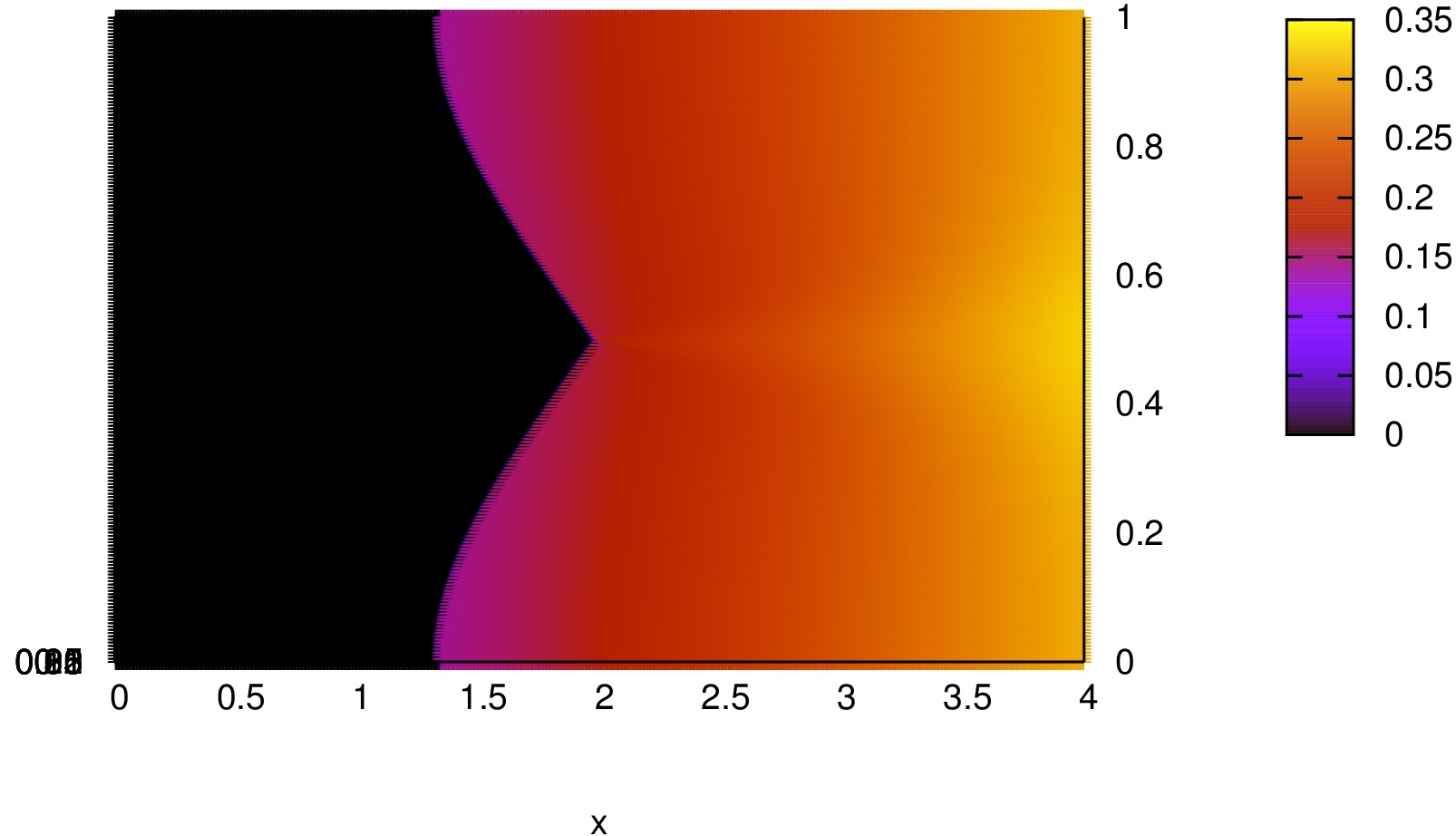} &
\includegraphics[width=7cm]{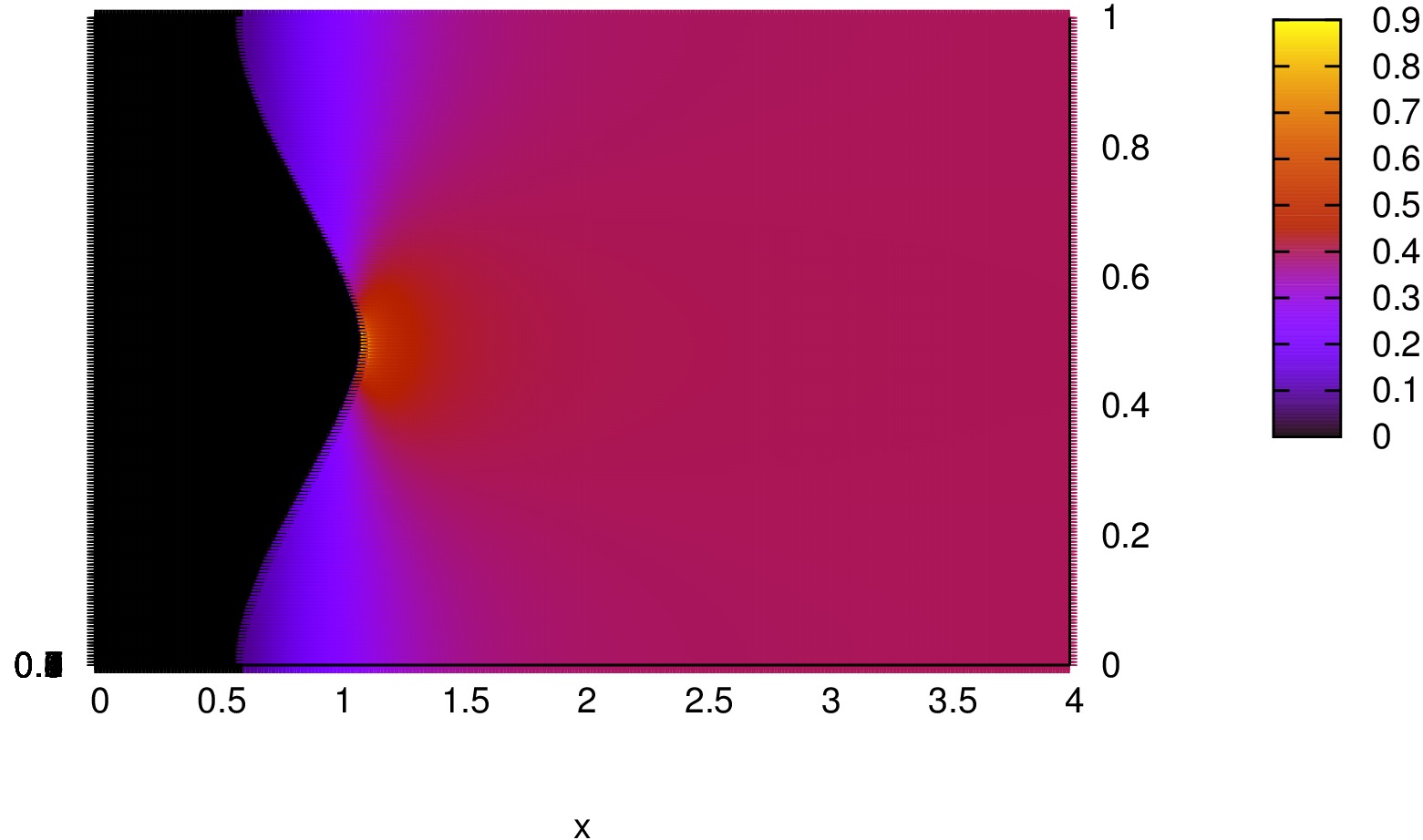}\\
\includegraphics[width=7cm]{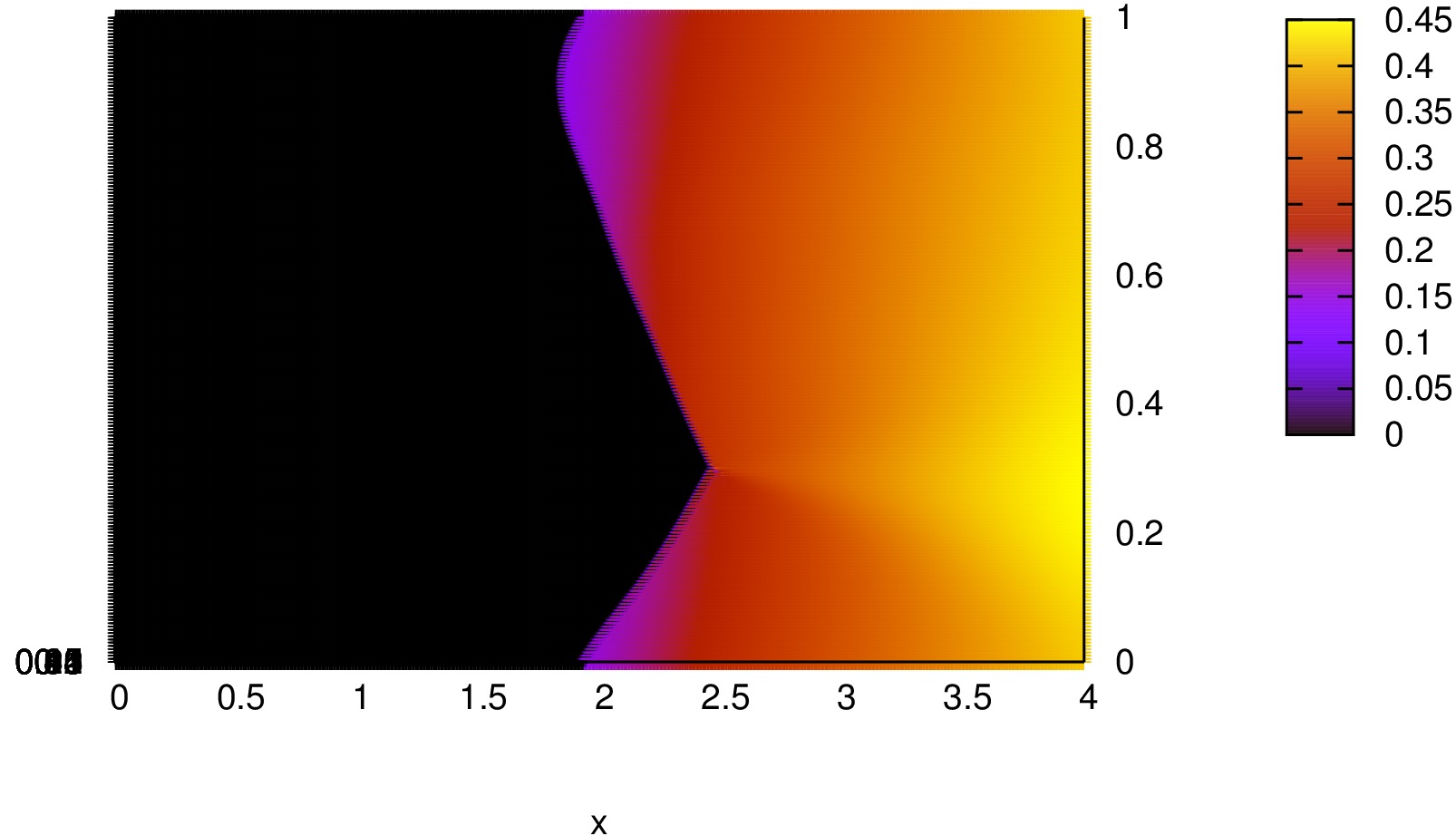} &
\includegraphics[width=7cm]{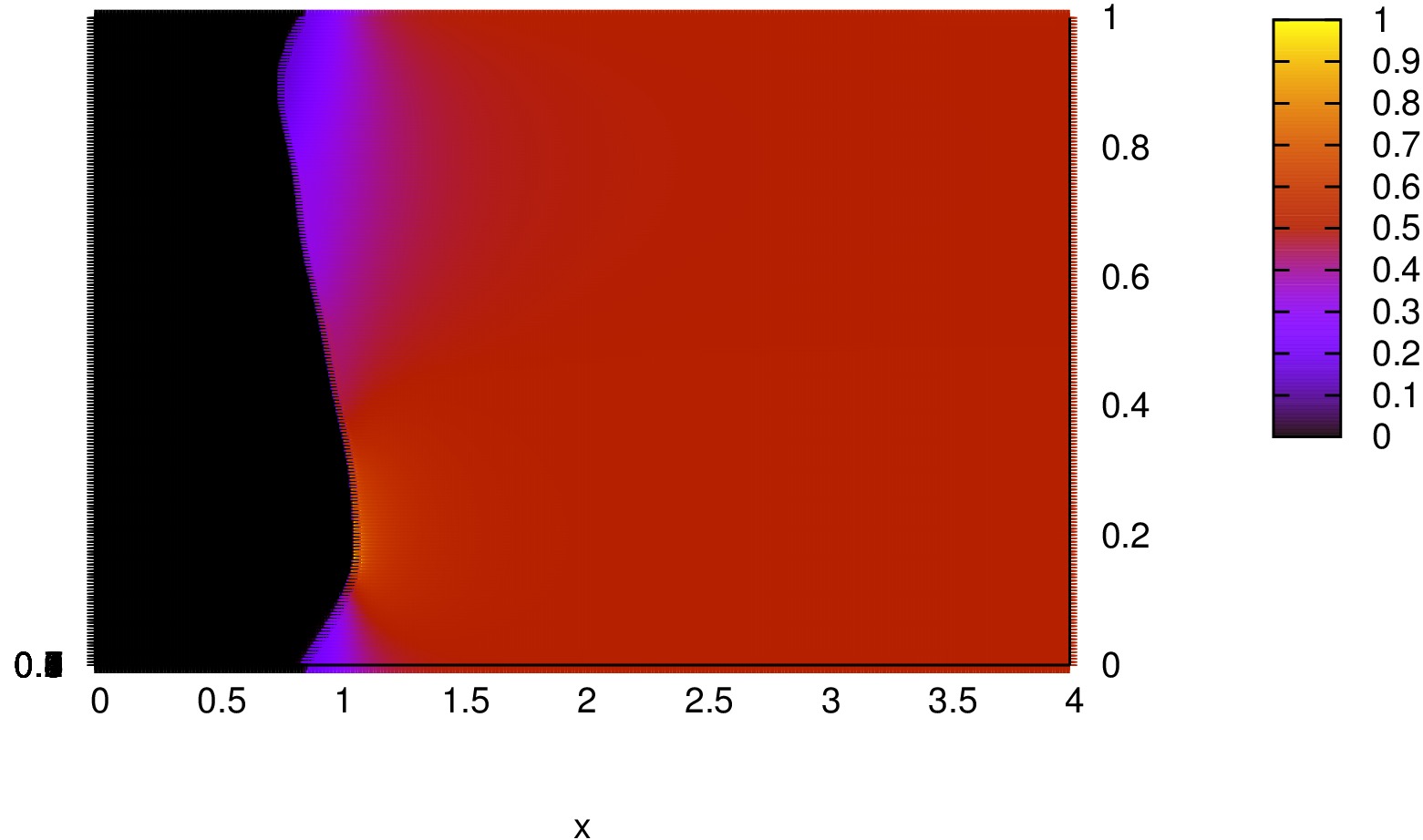}
\end{tabular}
\end{center}
\caption{View from top of $\partial_xp$ at time $t=30$ for $\alpha_1(y)$ with $c=0.6$ (top), $\alpha_2(y)$ with $c=0.5$ (middle), and $\alpha_3(y)$ with $c=0.4$ (bottom). The conductivity exponent is $m=0.1$ to the left, $m=1.1$ to the right, $\rd x=\rd y=5.e^{-3}$.}
\label{fig:corners_not_corners}
\end{figure}

Deciding qualitatively from numerical computations whether corners appear or not is a delicate matter, due to the intrinsic subjectivity of any graphical representation: what \emph{appears} to be corners in the plots of Figure~\ref{fig:corners_not_corners} (left column) may actually reveal to be smooth when zooming on the supposedly Lipschitz tips.
This does not seem to be the case, but the zooming possibilities are of course limited by the resolution and accuracy of the computations, hence this criterion is not very satisfactory.

Let us recall at this stage that \eqref{eq:HJ_0} was established in Theorem~\ref{theo:regularity} under the assumptions \eqref{hyp:H1}\eqref{hyp:H2}, which were numerically validated in section~\ref{subsection:num_validation_H1H2}.
As a consequence we can safely rely on the Hamilton-Jacobi scenario just discussed to push further the analysis: the forcing term $g(y)=\frac{c+\alpha(y)}{\left.\partial_xp\right|_{\Gamma^+}(y)}-1$ in $|\nabla_y \mathcal I |^2=g(y)$ can be evaluated numerically using Algorithm~\ref{algo:FB_detection_pxx} to compute (an approximation of) $\left.\partial_x p\right|_{\Gamma^+}$, and if $g$ does not vanish at a corner-\emph{looking} point $y_*\in \T$ then the free boundary $\Gamma=\{x=\mathcal I(y)\}$ \emph{must} have a Lipschitz corner there.
This is illustrated in Figure~\ref{fig:validation_HJ} in a test case: the forcing term is non-negative (as it should be, since it must equal $|\nabla_y \mathcal I |^2\geq 0$), vanishes only once at the $\mathcal{C}^1$ minimum $y\approx 0\equiv 1$ of $I$, and is clearly bounded away from zero in a $y$-neighborhood of the semiconcave corner at $y_*\approx 0.5$.
\begin{figure}[h!]
\begin{center}
\begin{tabular}{cc}
\includegraphics[width=7cm]{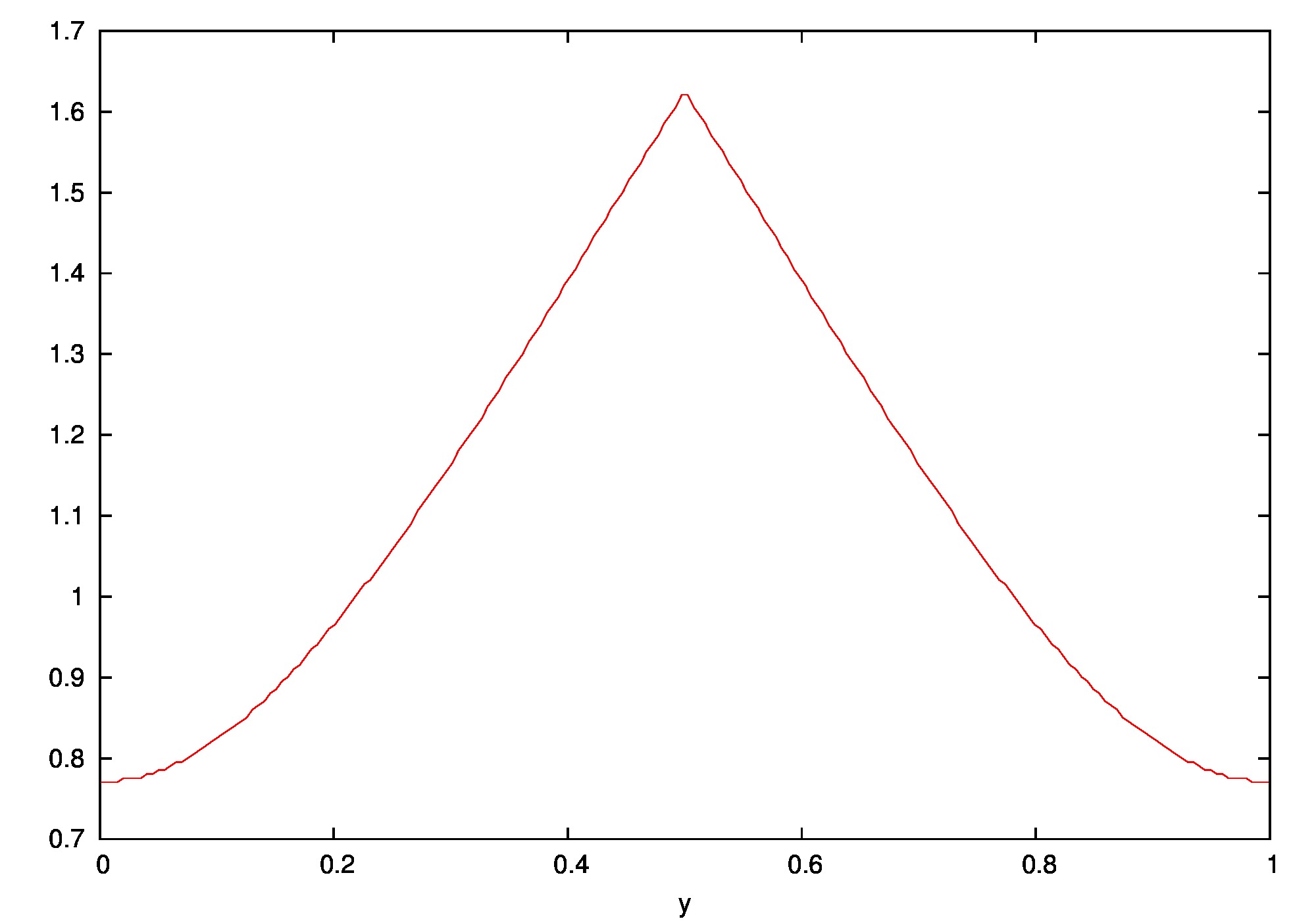}&
\includegraphics[width=7cm]{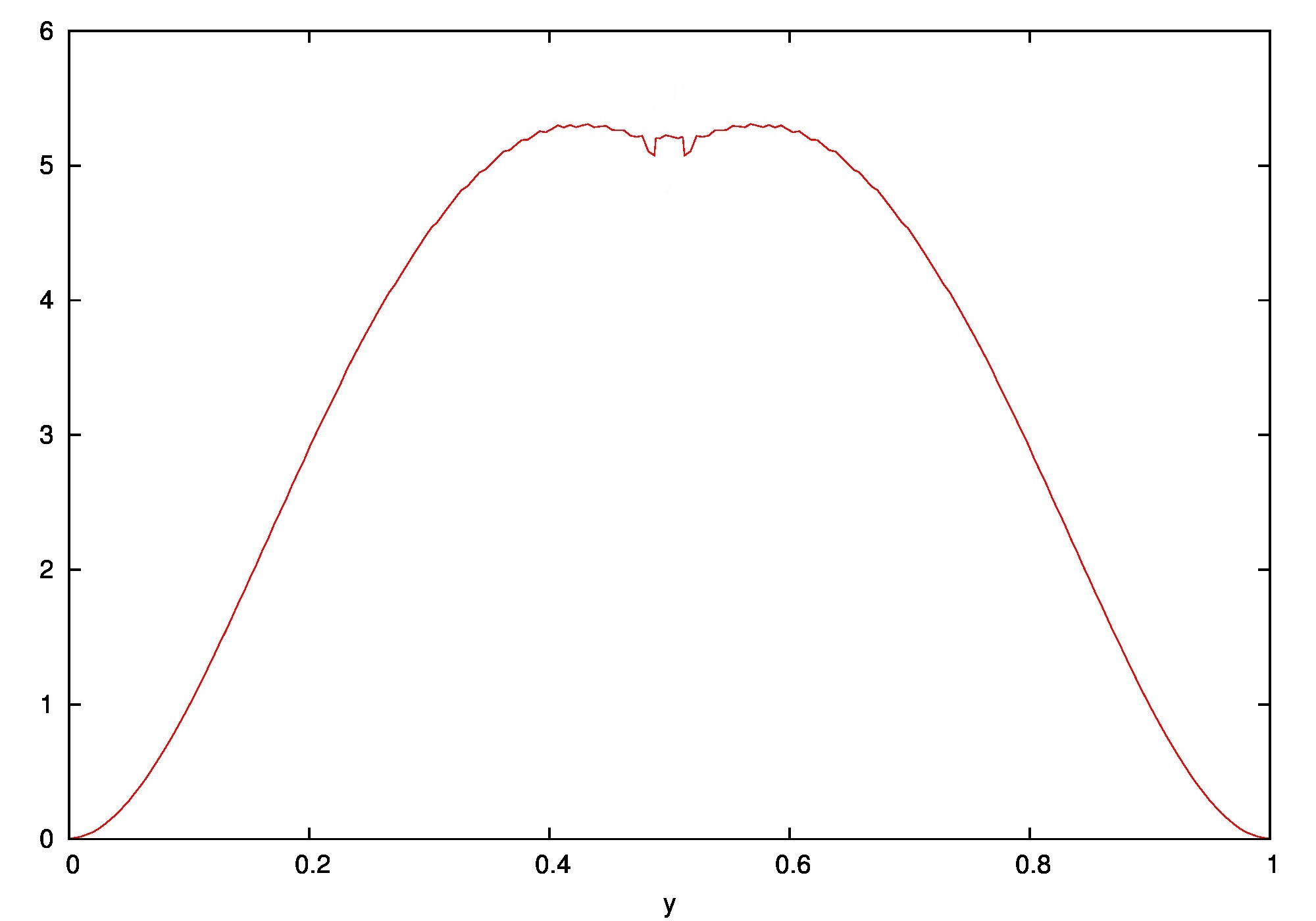}
\end{tabular}
\end{center}
\caption{Numerical validation of the corners. Interface parametrization $\mathcal I(y)$ to the left and forcing term $g(y)=\frac{c+\alpha(y)}{\left.\partial_xp\right|_{\Gamma^+}(y)}-1$ of \eqref{eq:HJ_0} to the right, plotted as functions of $y\in\T$.
The time $t=30$ is large enough so that the Cauchy solution has converged exponentially fast to the stationary wave profile. Parameters $m=0.1$, $\alpha(y)=\alpha_2(y)$, $c=0.4$, and $\rd x=\rd y=5.e^{-3}$.}
\label{fig:validation_HJ}
\end{figure}
%
%%%%%%%%%%%%%%%%%%%%%%%%%%%%%%%%%%%%%%%%%%%%%%%%%%%%%%%%%%%%%%%%%%%%%%%%%%%%%%%%%%%%%%%%%%%%%%%%%%%%
%
\subsection*{Acknowledgments}
The author wishes to thank Rapha\"el Loub\`ere for his help on the practical implementation, as well as Alexei Novikov and Jean-Michel Roquejoffre for fruitful discussions.
This work was supported by the French ANR project \emph{PREFERRED}, the Portuguese National Science Foundation through fellowship SFRH/
BPD/88207/2012, and the UT Austin/Portugal CoLab program \emph{Phase Transitions and Free Boundary Problems}
%
%

% \newpage
%\bibliographystyle{alpha}
\bibliographystyle{siam}
\bibliography{./biblio}

\end{document}